\documentclass[12pt,a4paper,intlimits,sumlimits]{amsart}

\usepackage{amsmath, amsthm, mathtools}
\usepackage{amsfonts}
\usepackage[alphabetic, nobysame]{amsrefs}
\usepackage{graphicx}
\usepackage{framed}
\usepackage{amssymb}
\usepackage{esvect}
\usepackage{xcolor}
\usepackage{eucal}
\usepackage{mathrsfs}
\usepackage[colorlinks=true,linkcolor=blue,urlcolor=blue]{hyperref}
\usepackage[english]{babel}
\usepackage{mathrsfs}
\usepackage{esint}

\def \N {\mathbb{N}}
\def \R {\mathbb{R}}

\theoremstyle{definition}
\newtheorem{definition}{Definition}[section]

\theoremstyle{plain}
\newtheorem{theorem}[definition]{Theorem}
\newtheorem{proposition}[definition]{Proposition}
\newtheorem{lemma}[definition]{Lemma}
\newtheorem{corollary}[definition]{Corollary}

\renewcommand{\emptyset}{\varnothing}

\numberwithin{equation}{section}

\usepackage{geometry}
\renewcommand{\epsilon}{\varepsilon}
\newcommand{\e}{\varepsilon}

\renewcommand{\leq}{\leqslant}
\renewcommand{\le}{\leqslant}
\renewcommand{\geq}{\geqslant}
\renewcommand{\ge}{\geqslant}

\allowdisplaybreaks

\title[Nonlocal eigenvalue problems]{Nonlocal eigenvalue problems \\
and superposition operators}

\author{Serena Dipierro, Edoardo Proietti Lippi, Caterina Sportelli and Enrico Valdinoci}

\address{Serena Dipierro: Department of Mathematics and Statistics, The University of Western Australia, 35 Stirling Highway, Crawley, Perth, WA 6009, Australia}
\email{serena.dipierro@uwa.edu.au}

\address{Edoardo Proietti Lippi: Department of Mathematics and Statistics, The University of Western Australia, 35 Stirling Highway, Crawley, Perth, WA 6009, Australia}
\email{edoardo.proiettilippi@uwa.edu.au}

\address{Caterina Sportelli: Departamento de Análisis Matemático, Universidad de Granada, 18071 Granada, Spain \\\&\\ Department of Mathematics and Statistics, The University of Western Australia, 35 Stirling Highway, Crawley, Perth, WA 6009, Australia}
\email{caterina.sportelli@uwa.edu.au, caterina.sp@ugr.es}

\address{Enrico Valdinoci: Department of Mathematics and Statistics, The University of Western Australia, 35 Stirling Highway, Crawley, Perth, WA 6009, Australia}
\email{enrico.valdinoci@uwa.edu.au}

\begin{document}

\maketitle

\begin{abstract}
We study the spectral theory of mixed local and nonlocal operators with lower-order terms in the right-hand side of the equation. 

This kind of problems is motivated by the
analysis of superposition operators of mixed order and with the ``wrong sign''
of the lower-order terms
with respect to the classical elliptic theory.

Our results include: 
\begin{itemize}
\item convergence to classical cases when the right-hand side of the eigenvalye equations ``localizes'', recovering the simplicity and sign-definiteness of eigenfunctions in the limit;
\item a detailed analysis of disconnected domains, showing that, unlike the classical case, any eigenfunction associated with the first eigenvalue must change sign, and that the first eigenvalue of a union of disconnected domains is strictly smaller than that of its individual components;
\item examples in which the first eigenvalue is either simple or non-simple in disconnected domains;
\item a regularity theory that underpins these results. \end{itemize}
\end{abstract}

\tableofcontents

\section{Introduction} 

This article considers superpositions of linear operators of fractional order. Such superpositions may be either discrete or continuous (that is, they may involve finitely or infinitely many operators, possibly distributed according to a measure). Moreover, lower-order operators are allowed to appear with the ``wrong'' sign. Namely, while the higher-order operators in the superposition exhibit an elliptic structure, the lower-order ones may appear with the opposite sign.

This feature is particularly interesting from a modeling perspective, as it allows one to describe the overlap of diffusive phenomena (possibly subject to anomalous diffusion of fractional type) with singularity formation or concentration effects (arising from the presence of elliptic operators with the opposite sign).

The study of these superposition operators naturally leads to certain nonlocal equations reminiscent of eigenvalue problems. A notable structural difference with respect to the classical setting, however, is that the right-hand side of these equations does not consist solely of the function itself multiplied by the corresponding eigenvalue. Instead, it involves a superposition of lower-order operators carrying the ``wrong'' sign.

While there is an extensive literature on nonlocal superposition operators
(see e.g.~\cite{MR2129093, MR2244602, MR2542727, MR2653895, MR2911421, MR2912450, MR3485125, MR4387204, MR4793906, MR4391102, DPSV2, MR4845869},  \cite{TUTTI, TUTTI2, TUTTI3}), also in view
of problems arising in mathematical biology and population dynamics
(see~\cite{MR4651677, MR4249816, DPLSVlog}),
the study of eigenvalue problems of this type appears to be new, to the best of our knowledge. 

We also recall that the analysis of fractional eigenvalue problems,
even in the presence of a single nonlocal operator, already presents a number of unexpected challenges, 
see for example~\cite{KASSIL}.
Also, nonlocal eigenvalue problems typically present interesting features in disconnected domains that are structurally different from their nonlocal counterpart. For example, the sum of eigenfunctions supported
in two different connected components does not produce an eigenfunction in the union of the connected components (not even if the eigenvalue is the same for both connected component, nor if the two connected components are congruent, the reason being that the fractional Laplacian of the eigenfunction of one component does not necessarily vanish in the other connected component).
\medskip

In this paper, our main focus is the spectral theory of mixed local and nonlocal operators, with a superposition of lower-order terms in the right-hand side of the equation. We obtain:
\begin{itemize}
\item results in the spirit of perturbation theory, stating the convergence to classical cases when the right-hand side of the equation ``localizes'' (Theorem~\ref{main1}), and recovering the simplicity of the eigenvalues and the sign-definiteness of the eigenfunction from the limit case (Theorem~\ref{simple>0});
\item a thorough analysis of the case of disconnected domains, establishing that, differently from the classical case, any eigenfunction associated with the first eigenvalue must change sign (Theorem~\ref{propcambiosegno}), and that the first eigenvalue in the union of disconnected domains is strictly smaller than that of its individual components (Theorem~\ref{ATTEHAMS});
\item also, in the case of disconnected domains, examples in which the first eigenvalue is or is not simple (Theorems~\ref{lambdasnonsemplice} and~\ref{lambdassemplice});
\item a convenient regularity theory used in the proofs of the above results (Theorem~\ref{thmregolarità}).
\end{itemize}
\medskip

The specific mathematical framework in which we work is as follows.
Let $\mu^+$ and $\mu^-$ denote two nonnegative finite Borel measures over the interval~$[0, 1]$. 
We consider the signed measure
\begin{equation}\label{misure+-}
\mu:=\mu^+ -\mu^-.
\end{equation}
The linear superposition of (possibly) fractional operators that we take into account in this paper
occurs through operators of the form
\begin{equation}\label{superpositionoperator}
L_\mu u:=\int_{[0, 1]} (-\Delta)^s u\, d\mu(s),
\end{equation}

As customary, the notation~$(-\Delta)^s$ is reserved to the fractional Laplacian, defined, for all~$s\in(0,1)$, as
\begin{equation}\label{AMMU0} (- \Delta)^s\, u(x) = c_{N,s}\int_{\R^N} \frac{2u(x) - u(x+y)-u(x-y)}{|y|^{N+2s}}\, dy.
\end{equation}
The positive normalizing constant~$c_{N,s}$, defined by
\begin{equation}\label{defcNs}
c_{N,s}:=\frac{2^{2s-1}\,\Gamma\left(\frac{N+2s}{2}\right)}{\pi^{N/2}\,\Gamma(2-s)}\, s(1-s),
\end{equation}
is chosen in such a way that, for~$u$ smooth and rapidly decaying, the Fourier transform of~$(-\Delta)^s u$ returns~$(2\pi|\xi|)^{2s}$ times the Fourier transform of~$u$ and provides consistent limits as~$s\nearrow1$ and as~$s\searrow0$, namely
\[
\lim_{s\nearrow1}(-\Delta)^su=(-\Delta)^1u=-\Delta u
\qquad{\mbox{and}}\qquad
\lim_{s\searrow0}(-\Delta)^s u=(-\Delta)^0u=u,
\]
see e.g.~\cite{MR2944369}. 

We assume that there exist~$\overline s \in (0, 1]$ such that
\begin{equation}\label{mu0}
\mu^+\big([\overline s, 1]\big)>0
\qquad \mbox{and} \qquad
\mu^-\big([\overline s, 1]\big)=0.
\end{equation}
Roughly speaking, condition~\eqref{mu0} states that the higher values of the interval~$[0,1]$
(specifically, the ones above~$\overline s$)
are contained in the support of positive component of the measure~$\mu$.

Given a bounded open set~$\Omega$ with boundary of class~$C^1$,
our first set of results deals with the type problem
\begin{equation}\label{eigenvalueproblemmu}
\begin{cases}
\displaystyle\int_{[0, 1]} (-\Delta)^s u\, d\mu^+(s) = \lambda \int_{[0, \overline{s})} (-\Delta)^s u\, d\mu^-(s)  &  \mbox{ in } \Omega,
\\
u=0 &\mbox{ in } \R^N \setminus \Omega.
\end{cases}
\end{equation}

This can be seen as a ``doubly-nonlocal'' eigenvalue problem,
with the right-hand-side containing operators of ``lower order''
(in the classical eigenvalue problems, the right-hand side contains
simply an operator of order zero).

To start with, we establish the existence of a positive eigenvalue for problem \eqref{eigenvalueproblemmu}. In order to state our result, we introduce the space $X^+(\Omega)$ defined as the completion
of the compactly supported smooth functions with respect to a norm which involves the measure $\mu^+$ (for more details see the forthcoming definitions in \eqref{normadefinizione} and \eqref{X+Omegadefn}). In this setting, we have:

\begin{proposition}\label{lambda1}
Assume that \eqref{mu0} holds.

Then, problem~\eqref{eigenvalueproblemmu} admits a first
positive eigenvalue $\lambda_{\mu}(\Omega)$ satisfying
\begin{equation}\label{lambdamuOmega}
\lambda_{\mu}(\Omega):= \min_{u\in X^+(\Omega)\setminus\{0\}} \dfrac{\displaystyle\int_{[0, 1]} [u]^2_s \, d\mu^+(s)}{\displaystyle \int_{[0, \overline{s})} [u]^2_s \, d\mu^-(s)}.
\end{equation}
\end{proposition}

In order to stete our next result, we introduce some additional notations.
Let $\lambda_0\in\R$ denote the first eigenvalue of the problem
\begin{equation}\label{lambda0u0819e8ufjvo}
\begin{cases}
\displaystyle\int_{[0, 1]} (-\Delta)^s u\, d\mu^+(s) = \lambda u&  \mbox{ in } \Omega,\\
u=0 &  \mbox{ in } \R^N \setminus \Omega.
\end{cases}
\end{equation}
If~$u_0$ is one of the corresponding eigenfunctions and we use the normalization~ $\|u_0\|_{L^2(\Omega)}=1$, the variational characterization of $\lambda_0$ gives
\begin{equation}\label{lambda0u0}
\lambda_0:=\min_{u\in X^+(\Omega)\setminus\{0\}} \dfrac{\displaystyle\int_{[0, 1]} [u]^2_s \, d\mu^+(s)}{\displaystyle \int_\Omega |u(x)|^2\,dx}=\int_{[0, 1]} [u_0]^2_s \, d\mu^+(s).
\end{equation}
Now, we let $\varepsilon_n$ be a decreasing sequence such that 
\begin{equation}\label{epsnto0}
\lim_{n\to +\infty} \varepsilon_n =0.
\end{equation}
In addition, we consider a sequence $\mu^-_n$ of finite Borel measures over the set $[0, 1]$ such that for any $n\in\N$,
\begin{align}\label{successionemu-1}
& \mbox{supp}(\mu^-_n) \subset [0, \varepsilon_n]\\ \label{successionemu-2}
\mbox{and}\quad &\mu^-_n ([0, \varepsilon_n]) = 1.
\end{align}
In a nutshell, we are supposing here that~$\mu^-_n$ is a sequence of probability measures
supported in a small right interval of the origin.
In light of 
\eqref{successionemu-1}, \eqref{successionemu-2} and
Proposition~\ref{lambda1},  for any $n\in \N$ we can consider the eigenvalue problem
\begin{equation}\label{probsuccessione}
\begin{cases}
\displaystyle\int_{[0, 1]} (-\Delta)^s u\, d\mu^+(s) = \lambda \int_{[0, \varepsilon_n]} (-\Delta)^s u\, d\mu_n^-(s)  &  \mbox{ in } \Omega,
\\
u=0    &  \mbox{ in } \R^N \setminus \Omega.
\end{cases}
\end{equation}
For any $n\in\N$, we denote by $\lambda_{\varepsilon_n}$ the first eigenvalue associated with problem \eqref{probsuccessione} and by $u_{\varepsilon_n}$ one of the eigenfunctions associated with~$\lambda_{\varepsilon_n}$.
Up to a normalization, we can assume that 
\begin{equation}\label{lambdaen1}
\int_{[0, \varepsilon_n]}[u_{\varepsilon_n}]^2_s\,d\mu^-_n(s)=1.
\end{equation}
Hence, the variational characterization of $\lambda_{\varepsilon_n}$ gives
\begin{equation}\label{lambdaen2}
\lambda_{\varepsilon_n}:=\min_{u\in X^+(\Omega)\setminus\{0\}} \dfrac{\displaystyle\int_{[0, 1]} [u]^2_s \, d\mu^+(s)}{\displaystyle \int_{[0,\varepsilon_n] } [u]^2_s \, d\mu^-_n(s)}
=\int_{[0, 1]} [u_{\varepsilon_n}]^2_s \, d\mu^+(s).
\end{equation}

In this scenario, our next result gives some convergence properties for the first eigenvalue of problem \eqref{probsuccessione} and for its related eigenfunctions. It reads as follows:

\begin{theorem}\label{main1}
Assume that \eqref{mu0}, \eqref{successionemu-1} and \eqref{successionemu-2} hold. 

Let~$\lambda_0$ be the first eigenvalue of~\eqref{lambda0u0819e8ufjvo}.

Let also~$\lambda_{\varepsilon_n}$ and $u_{\varepsilon_n}$ be as in \eqref{lambdaen1} and \eqref{lambdaen2}.

Then, \begin{equation}\label{statement12-PRIMO}
\lim_{n\to +\infty} \lambda_{\varepsilon_n} =\lambda_0.\end{equation}

Furthermore,
there exists an eigenfunction~$u_0$ corresponding to the eigenvalue~$\lambda_0$ such that, up to a subsequence,
\begin{equation}\label{statement121P0EIDFJWEOBGML}
{\mbox{$u_{\varepsilon_n}$ converges weakly to $u_0$ in $X^+(\Omega)$}}\end{equation}
and
\begin{equation}\label{statement12}
{\mbox{$u_{\varepsilon_n}$ converges strongly to~$u_0$ in~$L^2(\Omega)$.}}
\end{equation}
\end{theorem}

The result in Theorem~\ref{main1} fits into the broad framework of perturbation theory for linear operators. We present a direct approach to its proof; more abstract versions can be obtained from the general theory (see for example~\cite{MR1335452}).

When the positive part of the measure $\mu$ is the Dirac delta centered in 1 (or more generally $\mu^+\big([0,1)\big)=0$), the space $X^+(\Omega)$ coincides with $H^1(\Omega)$ and problem \eqref{probsuccessione} boils down to
\begin{equation}\label{probsuccessionefacile}
\begin{cases}
\displaystyle -\Delta u= \lambda \int_{[0, \varepsilon_n]} (-\Delta)^s u\, d\mu_n^-(s)  &  \mbox{ in } \Omega,
\\
u=0    &  \mbox{ in } \R^N \setminus \Omega.
\end{cases}
\end{equation}
In this scenario, if $\Omega$ is connected,
by relying on Theorem \ref{main1} we can prove that the first eigenvalue found in 
Proposition~\ref{lambda1} is simple and the first eigenfunction can be 
taken to be positive. 
More explicitly, we have:

\begin{theorem}\label{simple>0}
Assume that \eqref{mu0}, \eqref{successionemu-1} and \eqref{successionemu-2} hold. Assume also that $\Omega$ is connected and that $\mu^+=\delta_1$. 

Then, there exists $n_0\in \N$ such that, for any $n>n_0$,
the first eigenvalue of problem \eqref{probsuccessionefacile}, defined as
\[
\lambda_{\varepsilon_n}:=\min_{u\in H^1(\Omega)\setminus\{0\}} \dfrac{\displaystyle \int_\Omega |\nabla u(x)|^2\,dx}{\displaystyle \int_{[0,\varepsilon_n] } [u]^2_s \, d\mu^-_n(s)},
\]
is simple. 

In addition, the eigenfunction $u_{\varepsilon_n}$ corresponding to $\lambda_{\varepsilon_n}$ does not change sign,
namely it is either strictly positive or strictly negative in $\Omega$.
\end{theorem}

The hypothesis that $\Omega$ is connected
in Theorem~\ref{simple>0} cannot be omitted.
Indeed, when $\mu^-\big((0,\varepsilon_n]\big)>0$, if $\Omega$ is disconnected, then any eigenfunction associated with the first eigenvalue must change
sign, according to the following result:

\begin{theorem}\label{propcambiosegno}
Let~$M\in\N\cap[2,+\infty)$.
Let $\Omega=\Omega_1\cup \cdots \cup \Omega_M$, where,
for any $j=1,\cdots,M$,
$\Omega_j$ is an open, bounded and connected set,
with boundary of class~$C^1$,
and such that~$\overline{\Omega_i}\cap\overline{ \Omega_j}=\emptyset$ if $i\neq j$.

Assume that $\mu^-\big((0,\overline s)\big)>0$ for some $\overline s\in (0,1)$.

Then, any eigenfunction associated with the first eigenvalue of the problem
\begin{equation}\label{xfgyxcfghxcfg}
\begin{cases}
-\Delta u= \lambda\displaystyle \int_{[0,\overline s)}(-\Delta)^s u\,d\mu^-(s) &\mbox{in }\Omega,
\\
u=0 &\mbox{in }\R^N \setminus \Omega,
\end{cases}
\end{equation}
changes sign (namely, it takes both strictly positive and strictly negative values).
\end{theorem}

The case of disconnected domains is interesting also because
it highlights some structural differences with respect to the classical case.
For instance, the first eigenvalue of the Laplacian in a disconnected domain
coincides with the smallest first eigenvalue among all the possible connected components (see Proposition~\ref{PROPD1}(i))
and if the smallest first eigenvalue
is reached in two connected components then the linear combinations
of the corresponding eigenfunctions in these components (extended
to zero outside) are eigenfunctions for the full domain (see
Proposition~\ref{PROPD1}(v)). This is not the case
in the local setting, as pointed out in the following result.
To state it, we use the notation
$$\lambda(\Omega,\mu^-):=
\min_{u\in H^1(\Omega)\setminus\{0\}} \dfrac{\displaystyle \int_\Omega |\nabla u(x)|^2\,dx}{\displaystyle \int_{[0,1) } [u]^2_s \, d\mu^-_n(s)}.$$

\begin{theorem}\label{ATTEHAMS}
Let~$\Omega_1$, $\Omega_2$ be bounded, connected, open sets,
with boundary of class~$C^1$,
and such that~$\overline{\Omega_1}\cap\overline{ \Omega_2}=\emptyset$.

Suppose that
\begin{equation}\label{0owejdmcwedf:w2ed3rfevrfs7mxwoevi}
\mu^-((0,1))\ne0.\end{equation}
Assume that, for~$j\in\{1,2\}$, 
\begin{equation}\label{0owejdmcwedf:w2ed3rfevrf}
{\mbox{the eigenvalue~$\lambda(\Omega_j,\mu^-)$
admits a positive eigenfunction~$u_{\Omega_j}$.}}\end{equation}

Then,
\begin{equation}\label{0owejdmcwedf:w2ed3rfevrf019kdcX3} \lambda(\Omega_1\cup\Omega_2,\mu^-)<\min\big\{
\lambda(\Omega_1,\mu^-),\lambda(\Omega_2,\mu^-)\big\}\end{equation}
and \begin{equation}\label{0owejdmcwedf:w2ed3rfevrf9}\begin{split}&{\mbox{a nonnegative linear combination of~$u_{\Omega_1}$ and~$u_{\Omega_2}$}}\\&{\mbox{is never an eigenfunction corresponding to~$ \lambda(\Omega_1\cup\Omega_2,\mu^-)$.}}\end{split}\end{equation}
\end{theorem}

We stress that hypothesis~\eqref{0owejdmcwedf:w2ed3rfevrf} is fulfilled,
for example, in the setting of Theorem~\ref{simple>0}.

A related question deals with the simplicity of the first eigenvalue.
Quite surprisingly, for disconnected domains this eigenvalue may or may not be simple. This feature is displayed in the following results:

\begin{theorem}\label{lambdasnonsemplice}
Let $\Omega:=B_\frac{1}{2}(e_1)\cup B_\frac{1}{2}(-e_1)$. Denote by $\lambda_s$ the first eigenvalue of problem
\[
\begin{cases}
-\Delta u=\lambda (-\Delta)^s u &\mbox{in }\Omega,
\\
u=0  &\mbox{in }\R^N\setminus \Omega.
\end{cases}
\]

Then, $\lambda_s$ is not simple provided that $s$ is small enough.
\end{theorem}

\begin{theorem}\label{lambdassemplice}
Let $\Omega:=B_\frac{1}{2}(e_1)\cup B_\frac{1}{4}(-\frac{1}{2}e_1)$. Denote by $\lambda_s$ the first eigenvalue of problem
\[
\begin{cases}
-\Delta u=\lambda (-\Delta)^s u &\mbox{in }\Omega,
\\
u=0  &\mbox{in }\R^N\setminus \Omega.
\end{cases}
\]

Then, $\lambda_s$ is simple provided that $s$ is small enough.
\end{theorem}

Regarding the proof of Theorem \ref{simple>0}, 
in the course of its proof
we need
an auxiliary regularity result, which may have an independent interest.
To state it, it is convenient to introduce a positive number~$\overline\mu$ such that
$$\overline\mu\ge \mu^+([0,1])+\mu^-([0,1]).$$
In this setting, we have:

\begin{theorem}\label{thmregolarità}
Let~$\widetilde s\in[0,1/2)$.
Assume that $f\in L^p(\Omega)$ with $p\in(1,+\infty)$. 

Let $u\in H^1_0(\Omega)$ be a weak solution\footnote{The notion of weak solution in this paper is understood in the sense of testing the equation against smooth functions compactly supported in~$\Omega$.} of the problem
\begin{equation}\label{problemaregolarità}
\begin{cases}
-\Delta u +\displaystyle\int_{[0,\widetilde s]}(-\Delta)^s u\,d\mu(s)
=f(x) & \mbox{in } \Omega,
\\
u=0  & \mbox{in } \R^N\setminus \Omega.
\end{cases}
\end{equation}

Then, $u\in W^{2,p}(\Omega)$ and there exists a constant $C=C(N,p,\Omega,\widetilde s,\overline\mu)$ such that:
\begin{equation}\label{inequalityregolarità}
{\mbox{when either $N=1$ or $p\in[2,+\infty)$,}}\qquad
\|u\|_{W^{2,p}(\Omega)}\leq 
C\,\Big( \|f\|_{L^p(\Omega)}+\|u\|_{L^2(\Omega)}\Big)
\end{equation}
and
\begin{equation}\label{inequalityregolaritàlow}
{\mbox{when $p\in(1,2)$,}}\qquad
\|u\|_{W^{2,p}(\Omega)}\leq 
C\,\Big( \|f\|_{L^p(\Omega)}+\|u\|_{L^p(\Omega)}\Big).
\end{equation}
\end{theorem}

We will give in this paper a self-contained proof of Theorem~\ref{thmregolarità}
(see e.g.~\cite[Theorem~3.1.19]{MR1911531}
for related results in the class of strong solutions).\medskip

The rest of the paper is organized as follows. After discussing
some preliminary convergence properties in Section~\ref{oajdcfwe9rtm92-34v50nVGm504r02390},
we address the regularity theory needed in this paper in Section~\ref{oajdcfwe9rtm92-34v50nVGm504r023902},
where we prove Theorem~\ref{thmregolarità}.

Then, in Section~\ref{0qspidkfewp0rthn9nf5ac01sklO}
we discuss the perturbation theory of nonlocal eigenvalue problems,
establishing Theorems~\ref{main1} and \ref{simple>0}.

The case of disconnected domains is dealt with in Section~\ref{sezionedisconnessi},
which contains the proofs of
Theorems~\ref{propcambiosegno}, \ref{ATTEHAMS},
\ref{lambdasnonsemplice}
and~\ref{lambdassemplice}.

The paper ends with an appendix collecting auxiliary observations
of technical flavor, as well as the proof of Proposition~\ref{lambda1}.

\section{Preliminaries: functional spaces and basic convergence properties}\label{oajdcfwe9rtm92-34v50nVGm504r02390}
In this section we introduce our functional analytical setting and we gather some preliminary observations.

To this end, for~$s\in[0,1]$, we let
\begin{equation}\label{seminormgagliardo}
[u]_s:=
\begin{cases}
\|u\|_{L^2(\R^N)}  &\mbox{ if } s=0,
\\ \\
\displaystyle\left(c_{N,s}\iint_{\R^{2N}}\frac{|u(x)-u(y)|^2}{|x-y|^{N+2s}}\,dx\,dy \right)^{1/2} &\mbox{ if } s\in(0,1),
\\ \\
\|\nabla u\|_{L^2(\R^N)}  &\mbox{ if } s=1.
\end{cases}
\end{equation}
We define 
\begin{equation}\label{normadefinizione}
\|u\|_{X^+}:=\left(\;\int_{[0,1]}[u]^2_s\,d\mu^+(s)\right)^{\frac12}.
\end{equation}
Then, for any open and bounded set~$\Omega\subset\R^N$ with Lipschitz boundary, we define the space 
\begin{equation}\label{X+Omegadefn}
X^+(\Omega) \mbox{ as the completion of $C^\infty_c(\Omega)$ with respect to $\|\cdot\|_{X^+}$},
\end{equation}
and refer the interested reader to \cite{DPLSVlog} for more details on this functional space.

We remark that $\|\cdot\|_{X^+}$ is a norm in the space $X^+(\Omega)$ and
\begin{equation}\label{X+Hilbert}
\mbox{$X^+(\Omega)$ is a Hilbert space with respect to the norm $\|u\|_{X^+}$},
\end{equation}
which is endowed with the scalar product defined, for any~$u$, $v\in X^+(\Omega)$, as
\begin{equation}\label{scalarepiu}
\langle u, v\rangle_\pm :=  \int_{[0,1]}\langle u, v\rangle_s\,d\mu^\pm(s),
\end{equation}
where
$$\langle u, v\rangle_s:=
\begin{cases}
\displaystyle\int_\Omega u(x)\,v(x)\,dx  &\mbox{ if } s=0,
\\ \\
c_{N,s}\displaystyle\iint_{\R^{2N}}\frac{(u(x)-u(y))(v(x)-v(y))}{|x-y|^{N+2s}}\,dx\,dy &\mbox{ if } s\in(0,1),
\\ \\
\displaystyle\int_\Omega \nabla u(x)\cdot\nabla v(x)\,dx  &\mbox{ if } s=1.
\end{cases}
$$

In our analysis, we will make use of the following result of Sobolev type, which states that higher exponents in fractional norms control the lower exponents with uniform constants (see~\cite[Lemma~2.1]{MR4736013}).
\begin{lemma}\label{nons} 
Let~$0\le s_1 \le s_2 \le1$. 

Then, for any measurable function~$u:\R^N\to\R$ with~$u=0$ a.e. in~$\R^N\setminus\Omega$ we have that
\begin{equation}\label{spp}
[u]_{s_1}\le C \, [u]_{s_2},
\end{equation}
for a suitable positive constant~$C=C(N,\Omega)$.
\end{lemma}

We now recall a particular case of the result stated in~\cite[Proposition~2.4]{MR4736013}, which provides some useful embeddings for the space $X^+(\Omega)$ defined in \eqref{X+Omegadefn}.

\begin{lemma}\label{embeddingsXOmega}
The space $X^+(\Omega)$ is continuously embedded in $H^{\overline s}(\Omega)$.

Furthermore, $X^+(\Omega)$ is continuously and compactly embedded in~$L^2(\Omega)$.
\end{lemma}

Now we present some preliminary results.

\begin{lemma}\label{lemma2}
Assume that \eqref{mu0}, \eqref{successionemu-1}, \eqref{successionemu-2} hold. Let $u_n$ be a sequence such that $u_n$ converges weakly in $X^+(\Omega)$ to some $u_*$.

Then,
\begin{equation}\label{limitemu-n2}
\lim_{n\to +\infty}\left|\,\int_{[0, \varepsilon_n]} [u_n]^2_s\, d\mu^-_n(s) - \|u_*\|^2_{L^2(\R^N)} \right| = 0.
\end{equation}
\end{lemma}
\begin{proof}
Since~$u_n$ converges weakly in $X^+(\Omega)$, by Lemma \ref{embeddingsXOmega} we deduce that
\begin{equation}\label{convL2}
{\mbox{$u_n\to u_*$ in~$L^2(\R^N)$ as $n\to +\infty $.}}
\end{equation}
From this and the Plancherel Theorem, we get
\begin{equation}\label{convL2trasformata}
{\mbox{$\widehat{u_n}\to \widehat{u}_*$ in $L^2(\R^N)$ as $n\to +\infty $.}}
\end{equation}
Hence,  \cite[Theorem 4.9]{MR2759829} ensures that
\begin{equation}\label{hinL2}
\mbox{there exists $h\in L^2(\Omega)$ such that } |\widehat{u_n}(x)|\le h(x) \ \mbox{ for any $n\in\N$, a.e. in $\Omega$}.
\end{equation}
Moreover, we notice that
\begin{eqnarray*}&&
\left|\,\int_{[0, \varepsilon_n]} [u_n]^2_s\, d\mu^-_n(s) - \|u_*\|^2_{L^2(\R^N)} \right|\\&&\qquad\le \left|\,\int_{[0, \varepsilon_n]} [u_n]^2_s\, d\mu^-_n(s) - \|u_n\|^2_{L^2(\R^N)} \right| + \left|\,\|u_n\|^2_{L^2(\R^N)} - \|u_*\|^2_{L^2(\R^N)} \right|.
\end{eqnarray*}
From this and \eqref{convL2},  we deduce that, to obtain~\eqref{limitemu-n2}, it is enough to show that
\begin{equation}\label{limitemu-n3}
\lim_{n\to +\infty}\left|\,\int_{[0, \varepsilon_n]} [u_n]^2_s\, d\mu^-_n(s) - \|u_n\|^2_{L^2(\R^N)} \right| = 0.
\end{equation}
We check \eqref{limitemu-n3}. To this aim, it is convenient to write the Gagliardo seminorm in terms of the Fourier transform (see e.g.~\cite{MR2944369}). Indeed, for any $n\in\N$,
\[
[u_n]_s=\left(\;\int_{\R^{N}} |2\pi \xi|^{2s}\,|\widehat{u_n}(\xi)|^2\,d\xi\right)^{\frac12}.
\]
We stress that, by the Plancherel Theorem, this is also valid when $s=0$ and $s=1$.

Consequently, we have
\[
\begin{split}
&\left|\,\int_{[0, \varepsilon_n]} [u_n]^2_s\, d\mu^-_n(s) - \|u_n\|^2_{L^2(\R^N)} \right| = \left|\,\int_{[0, \varepsilon_n]} \left(\,\int_{\R^N}|\widehat{u_n}(\xi)|^2 (|2\pi\xi|^{2s}-1) d\xi \right) d\mu^-_n(s)\right|\\
&\qquad\le \int_{\R^N} |\widehat{u_n}(\xi)|^2 \left(\,\int_{[0, \varepsilon_n]} \big| |2\pi\xi|^{2s}-1\big| d\mu^-_n(s)\right) d\xi.
\end{split}
\]
We take $\overline s$ as in \eqref{mu0} and we let $\delta>0$ be small enough.  We observe that the function
\[
|\xi|\in\R^N\setminus B_\delta \mapsto \frac{\big| |2\pi\xi|^{2s}-1\big|}{|2\pi\xi|^{2\overline s}} \mbox{ has a maximum at } \frac{1-(2\pi\delta)^{2s}}{(2\pi\delta)^{2\overline s}}.
\]
From this, \eqref{successionemu-2} and \eqref{hinL2}, we have
\begin{equation}\label{OlJSnwqm50vbX0-708}
\begin{split}
&\left|\,\int_{[0, \varepsilon_n]} [u_n]^2_s u\, d\mu^-_n(s) - \|u_n\|^2_{L^2(\R^N)} \right|\\
&\qquad\le\int_{B_\delta} |\widehat{u_n}(\xi)|^2 \left(\,\int_{[0, \varepsilon_n]} \big| |2\pi\xi|^{2s}-1\big| d\mu^-_n(s)\right) d\xi \\
&\qquad\qquad+ \int_{\R^N\setminus B_\delta} |\widehat{u_n}(\xi)|^2 |2\pi\xi|^{2\overline s}\left(\,\int_{[0, \varepsilon_n]} \frac{\big| |2\pi\xi|^{2s}-1\big|}{|2\pi\xi|^{2\overline s}} d\mu^-_n(s)\right) d\xi\\
&\qquad\le \int_{B_\delta} |\widehat{u_n}(\xi)|^2 d\xi + \left(\,\int_{[0, \varepsilon_n]}\frac{1-(2\pi\delta)^{2s}}{(2\pi\delta)^{2\overline s}} d\mu^-_n(s)\right)\int_{\R^N} |2\pi\xi|^{2\overline s} |\widehat{u_n}(\xi)|^2\, d\xi\\
&\qquad\le \int_{B_\delta} |h(\xi)|^2 d\xi + \frac{1-(2\pi\delta)^{2\varepsilon_n}}{(2\pi\delta)^{2\overline s}} [u_n]^2_{\overline s}.
\end{split}
\end{equation}

Besides, since $u_n$ is bounded in $X^+(\Omega)$, by Lemma \ref{nons} and \eqref{mu0}, we have that
\begin{eqnarray*}&&
[u_n]^2_{\overline s} \le \frac{C^2(N, \Omega)}{\mu^+([\overline s, 1])} \int_{[\overline s, 1]} [u_n]^2_s \, d\mu^+(s) \le\frac{ C^2(N, \Omega)}{\mu^+([\overline s, 1])} \int_{[0, 1]} [u_n]^2_s \, d\mu^+(s)\\&&\qquad\qquad\qquad\le \frac{ C^2(N, \Omega)}{\mu^+([\overline s, 1])}\sup_{n\in\N}
\|u\|_{X^+}<+\infty.
\end{eqnarray*}
Hence, we infer from \eqref{OlJSnwqm50vbX0-708} that
\[
\lim_{n\to +\infty}\left|\,\int_{[0, \varepsilon_n]} [u_n]^2_s u\, d\mu^-_n(s) - \|u_n\|^2_{L^2(\R^N)} \right| \le  \int_{B_\delta} |h(\xi)|^2 d\xi.
\]
Finally, by taking the limit as $\delta\searrow 0$, we obtain the desired result in \eqref{limitemu-n3}.
\end{proof}

\begin{corollary}\label{lemma1}
Assume that \eqref{mu0}, \eqref{successionemu-1}, \eqref{successionemu-2} hold and let $u\in X^+(\Omega)$. 

Then,
\begin{equation}\label{limitemu-n}
\lim_{n\to +\infty}\left|\,\int_{[0, \varepsilon_n]} [u]^2_s\, d\mu^-_n(s) - \|u\|^2_{L^2(\R^N)} \right| = 0.
\end{equation}
\end{corollary}

\begin{proof} It suffices to choose~$u_n:=u_*:=u$
in Lemma \ref{lemma2}.
\end{proof}

\begin{corollary}\label{cor=1}
Assume that \eqref{mu0}, \eqref{successionemu-1}, \eqref{successionemu-2} hold. Let $u_{\varepsilon_n}$ be a sequence satisfying~\eqref{lambdaen1}
and converging weakly in $X^+(\Omega)$ to some $u_*$.

Then, 
\[
\|u_*\|_{L^2(\Omega)} =1.
\]
\end{corollary}
\begin{proof}
By Lemma \ref{lemma2} and \eqref{lambdaen1}, we get
\[
1= \lim_{n\to +\infty}\int_{[0, \varepsilon_n]} [u_{\varepsilon_n}]^2_s u\, d\mu^-_n(s) = \|u_*\|_{L^2(\Omega)},
\]
as desired.
\end{proof}

\section{Some regularity results: proof of Theorem  \ref{thmregolarità}}\label{oajdcfwe9rtm92-34v50nVGm504r023902}
In this section we deal with the regularity theory in the nonlocal superposition setting and prove Theorem \ref{thmregolarità}.
To start with, we present the following preliminary results.

\begin{lemma}\label{misurainLp}
Let~$\widetilde s\in[0,1/2)$
and~$p\in(1,+\infty)$.
Let~$\Omega\subset\R^N$ be a bounded open set of class~$C^1$.

Then, for all\footnote{In our notation, functions
in~$W^{1,p}_0(\Omega)$ are implicitly assumed to be defined
in the whole of~$\R^N$, being null outside~$\Omega$. This is convenient to calculate the fractional Laplacian of these functions
without adding notational complications.} $u\in W^{1,p}_0(\Omega)$,
\[
\int_{[0,\widetilde s]}(-\Delta)^s u\,d\mu(s)\in L^p(\Omega).
\]
Moreover,  there exists a constant $C=C(N,\Omega,p,\widetilde s,\overline\mu)>0$ such that
\[
\left\|\,\int_{[0,\widetilde s]}(-\Delta)^s u\,d\mu(s) \right\|_{L^p(\Omega)}\leq C\|u\|_{W^{1,p}(\Omega)}.
\]
\end{lemma}

\begin{proof}
We first observe that, by the Minkowski inequality,
\begin{equation}\label{zserfgyuhnjkop1}
\left\|\,\int_{[0,\widetilde s]}(-\Delta)^s u\,d\mu(s) \right\|_{L^p(\Omega)}\leq \int_{[0,\widetilde s]}\|(-\Delta)^s u\|_{L^p(\Omega)}\,d\mu^+(s)+\int_{[0,\widetilde s]}\|(-\Delta)^s u\|_{L^p(\Omega)}\,d\mu^-(s).
\end{equation}
We now focus on giving a bound for $\|(-\Delta)^s u\|_{L^p(\Omega)}$.  For all $s\in (0, \widetilde s]$,  we can write
\begin{equation}\label{zserfgyuhnjkop2}
\begin{aligned}
\|(-\Delta)^s u\|_{L^p(\Omega)}&\leq \left(\int_\Omega\left|c_{N,s}\int_{\Omega\cap B_1(x)}\frac{u(x)-u(y)}{|x-y|^{N+2s}}dy\right|^p dx\right)^\frac{1}{p} \\
&\quad +\left(\int_\Omega\left|c_{N,s}\int_{(\R^N\setminus \Omega)\cap B_1(x)}\frac{u(x)-u(y)}{|x-y|^{N+2s}}dy\right|^p dx\right)^\frac{1}{p} \\
&\quad+\left(\int_\Omega\left|c_{N,s}\int_{\R^N\setminus B_1(x)}\frac{u(x)-u(y)}{|x-y|^{N+2s}}dy\right|^p dx\right)^\frac{1}{p} \\
&=:I_1+I_2+I_3.
\end{aligned}
\end{equation}
We first deal with $I_1$. 
To this end, by using the change of variables $z:=y-x$ and the H\"older inequality with conjugate exponents $p$ and $q=\frac{p}{p-1}$, we get
\begin{equation}\label{zserfgyuhnjkop3}
\begin{aligned}
I_1^p&\leq c_{N,s}^p\int_{\R^N}\left(\,\int_{B_1(0)}\frac{|u(x)-u(x+z)|}{|z|^{N+2s}}dz\right)^pdx \\
&\leq c_{N,s}^p\int_{\R^N}\left(\,\int_{B_1(0)}\int_0^1\frac{|\nabla u(x+\vartheta z)|}{|z|^{N+2s-1}}d\vartheta\,dz\right)^p dx \\
&\leq c_{N,s}^p \int_{\R^N}\left(\,\int_{B_1(0)}\int_0^1\frac{|\nabla u(x+\vartheta z)|^p}{|z|^{N+2s-1}}d\vartheta\,dz\right)\left(\,\int_{B_1(0)}\int_0^1\frac{1}{|z|^{N+2s-1}}d\vartheta\,dz\right)^\frac{p}{q} dx \\
&\leq c_{N,s}^p\left(\,\int_{B_1(0)}\frac{1}{|z|^{N+2s-1}}\,dz\right)^p \|u\|_{W^{1,p}(\R^N)}^p \\
&= \left(\frac{c_{N,s}\,\omega_{N-1}}{1-2s}\right)^p\|u\|_{W^{1,p}(\Omega)}^p.
\end{aligned}
\end{equation}

Now we deal with~$I_2$.
For any $x\in \Omega$, we denote by $x_0\in \partial \Omega$ a point such that $|x-x_0|=\operatorname{dist}(x,\partial\Omega)$. 
Notice that, for every~$y\in \R^N\setminus \Omega $, one has that~$|x-x_0|\le |x-y|$.
From this, recalling that $u\in W^{1,p}_0(\Omega)$, we get
\begin{equation}\label{zserfgyuhnjkop4}
\begin{aligned}
I_2^p&=c_{N,s}^p\int_\Omega\left|\,\int_{(\R^N\setminus \Omega)\cap B_1(x)}\frac{u(x)}{|x-y|^{N+2s}}dy\right|^p dx \\
&\le c_{N,s}^p\int_\Omega\left|\,\int_{ B_1(x)}\frac{u(x)}{|x-y|^{N+2s-1}|x-x_0|}dy\right|^p dx \\
&= c_{N,s}^p\int_\Omega\left|\,\int_{ B_1(0)}\frac{u(x)}{|z|^{N+2s-1}|x-x_0|}dz\right|^p dx \\
&= c_{N,s}^p\int_\Omega\left|\frac{u(x)}{|x-x_0|}\right|^p dx \ \left|\,\int_{B_1(0)}\frac{dz}{|z|^{N+2s-1}}dz\right|^p   \\
&=\left(\frac{c_{N,s}\,\omega_{N-1}}{1-2s}\right)^p\int_\Omega\left|\frac{u(x)}{|x-x_0|}\right|^p\,dx \\
&\leq \left(\frac{c_{N,s}\,\omega_{N-1}}{1-2s}\right)^p\int_{\R^N}\int_0^1 |\nabla u(x+\vartheta(x-x_0))|^p d\vartheta \,dx \\
&\leq \left(\frac{c_{N,s}\,\omega_{N-1}}{1-2s}\right)^p \|u\|_{W^{1,p}(\R^N)}^p\\&= \left(\frac{c_{N,s}\,\omega_{N-1}}{1-2s}\right)^p\|u\|_{W^{1,p}(\Omega)}^p.
\end{aligned}
\end{equation}

We now consider $I_3$.  By using
the H\"older inequality with exponents $p$ and $q=\frac{p}{p-1}$,  we obtain that
\begin{equation}\label{zserfgyuhnjkop5}
\begin{aligned}
I_3^p&\leq c_{N,s}^p\int_\Omega\left|\,\int_{\R^N\setminus B_1(x)}\frac{|u(x)|+|u(y)|}{|x-y|^{N+2s}}dy\right|^p dx \\
&\leq c_{N,s}^p\int_{\R^N}\left(\,\int_{\R^N\setminus B_1(0)}\frac{dz}{|z|^{N+2s}}\right)^\frac{p}{q}\left(\,\int_{\R^N\setminus B_1(0)}\frac{(|u(x)|+|u(x+z)|)^p}{|z|^{N+2s}}dz\right)dx \\
&\leq 2^{p-1}c_{N,s}^p\left(\,\int_{\R^N\setminus B_1(0)}\frac{dz}{|z|^{N+2s}}\right)^\frac{p}{q}\int_{\R^N}\int_{\R^N\setminus B_1(0)}\frac{|u(x)|^p+|u(x+z)|^p}{|z|^{N+2s}}dz\,dx \\
&\leq 2^pc_{N,s}^p\left(\,\int_{\R^N\setminus B_1(0)}\frac{dz}{|z|^{N+2s}}\right)^p \|u\|_{L^p(\R^N)}^p\\&\le\left(\frac{ c_{N,s}\,\omega_{N-1} }{s}\right)^p\|u\|_{W^{1,p}(\Omega)}^p.
\end{aligned}
\end{equation}

We now recall the setting in \eqref{defcNs} and define the constants
\[
c_1:=\max_{s\in[0,1]}c_{N,s} \quad \mbox{and} \quad c_2:=\max_{s\in[0,1]}\frac{c_{N,s}}{s}.
\]
Then, plugging \eqref{zserfgyuhnjkop3}, \eqref{zserfgyuhnjkop4} and \eqref{zserfgyuhnjkop5} into \eqref{zserfgyuhnjkop2}, we get
\begin{equation}\label{PKS35rik4301hgb:SXKD}
\|(-\Delta)^s u\|_{L^p(\Omega)}\leq \omega_{N-1}\left(\frac{2c_1}{1-2s}+c_2\right)\|u\|_{W^{1,p}(\Omega)}.
\end{equation}This estimate was obtained for all~$s\in (0, \widetilde s]$, but it is also obviously true when~$s=0$.
Accordingly, from~\eqref{zserfgyuhnjkop1} and~\eqref{PKS35rik4301hgb:SXKD}, we get
\[
\left\|\,\int_{[0,\widetilde s]}(-\Delta)^s u\,d\mu(s) \right\|_{L^p(\Omega)}\leq \omega_{N-1}\left(\frac{2c_1}{1-2\widetilde s}+c_2\right)(\mu^+([0,\widetilde s])+\mu^-([0,\widetilde s]))\|u\|_{W^{1,p}(\Omega)}.
\]
This concludes the proof.
\end{proof}

We are now ready to give the proof of Theorem \ref{thmregolarità}.

\begin{proof}[Proof of Theorem \ref{thmregolarità}] In a nutshell, we would like to combine
Lemma~\ref{misurainLp} with classical elliptic regularity theory,
by consider lower-order operator as a perturbation to be reabsorbed into
the right-hand side of the equation. When implementing this idea, however,
some care is needed depending on the exponents under consideration.

More specifically, we define
$$ \widetilde f:=f-\int_{[0,\widetilde s]}(-\Delta)^s u\,d\mu(s)$$
and we deduce from Lemma~\ref{misurainLp} that, for all~$P>1$,
\begin{equation}\label{PKSXe.dwf}
\|\widetilde f\|_{L^P(\Omega)}\le\|f\|_{L^P(\Omega)}
+
\left\|\,\int_{[0,\widetilde s]}(-\Delta)^s u\,d\mu(s) \right\|_{L^P(\Omega)}\leq \|f\|_{L^P(\Omega)}+C\|u\|_{W^{1,P}(\Omega)},\end{equation}
where we are, in principle, allowing some of the above norms to be infinite.

We also write~\eqref{problemaregolarità} in the form
\begin{equation*}
\begin{cases}
-\Delta u 
=\widetilde f(x) & \mbox{in } \Omega,
\\
u=0  & \mbox{in } \R^N\setminus \Omega.
\end{cases}
\end{equation*}
This allows us to apply classical elliptic results, such as~\cite[Theorem 5.3.5]{MR4784613}, and conclude that,
for all~$P>1$,
\begin{equation*}
\|u\|_{W^{2,P}(\Omega)}\leq C_1\,\|\widetilde f\|_{L^P(\Omega)},
\end{equation*}
for some constant~$C_1$, depending only on~$N$, $p$ and~$\Omega$.

This observation and~\eqref{PKSXe.dwf} give that, for all~$P>1$,
\begin{equation}\label{bg-92X4cv0 6m}
\|u\|_{W^{2,P}(\Omega)}\leq C_2\Big(
\|f\|_{L^P(\Omega)}+\|u\|_{W^{1,P}(\Omega)}\Big),
\end{equation}
for some constant~$C_2$, depending only on~$N$, $P$, $\Omega$,
$\widetilde s$ and~$\overline\mu$.

It is also helpful to recall that,
by an interpolation inequality (see e.g.~\cite[equation~(5.2.54)]{MR4784613}), given~$\e\in(0,1)$ we know that
\begin{equation}\label{INY} \|u\|_{W^{1,P}(\Omega)}\le \e\|u\|_{W^{2,P}(\Omega)}+C_\e\,\|u\|_{L^{P}(\Omega)},\end{equation}
for some~$C_\e$, depending only on~$N$, $P$, $\Omega$ and~$\e$.\bigskip

Now, let~$p$ be as in the statement of Theorem~\ref{thmregolarità}.
When~$p\in(1,2)$, we have that~$\|u\|_{W^{1,p}(\Omega)}$
is controlled by~$\|u\|_{H^1(\Omega)}$, which is finite by assumption,
hence~\eqref{bg-92X4cv0 6m}  (used with~$P:=p$)
leads to
\begin{equation}\label{08idk3vbbgy6ygcdr5m6o}
\|u\|_{W^{2,p}(\Omega)}\leq C_2\Big(
\|f\|_{L^p(\Omega)}+\|u\|_{W^{1,p}(\Omega)}\Big)<+\infty.\end{equation}
This and~\eqref{INY} give that
$$ \|u\|_{W^{2,p}(\Omega)}
\le
C_2\Big(
\|f\|_{L^p(\Omega)}+\e\|u\|_{W^{2,p}(\Omega)}+C_\e\,\|u\|_{L^{p}(\Omega)}\Big)
.$$One can now choose~$\e:=\frac{1}{2C_2}$, which allows one
to reabsorb the term~$\|u\|_{W^{2,p}(\Omega)}$
into the left-hand side and obtain~\eqref{inequalityregolaritàlow}
(we stress that to perform this reabsorption it is crucial that~$\|u\|_{W^{2,p}(\Omega)}$
is finite, and indeed~\eqref{08idk3vbbgy6ygcdr5m6o} ensures this).\bigskip

The proof of~\eqref{inequalityregolarità} is slightly more complicated, since, in general,
we do not know to start with that~$u\in W^{1,p}(\Omega)$ (and not even that~$u\in
L^p(\Omega)$); in this case, a recursive argument will be needed to get around this difficulty.

This recursive argument will be implemented on page~\pageref{APKsqwfe-r8ngtoglk4}, 
but, before that, let us make some general considerations.
First,  we remark that when~$N\le p$
we have that~$W^{1,p}(\Omega)$ is continuously contained in~$L^q(\Omega)$
for every~$q\in[1,+\infty)$, due to the Sobolev Embedding Theorem.
In particular, in this scenario, $\|u\|_{W^{1,p}(\Omega)}\le \widehat C\, \|u\|_{L^2(\Omega)}$, for some constant~$\widehat C$, depending only on~$N$, $p$ and~$\Omega$,
which, in tandem with~\eqref{bg-92X4cv0 6m},
retrurns that
\begin{equation*}
\|u\|_{W^{2,p}(\Omega)}\leq C_2\Big(
\|f\|_{L^p(\Omega)}+\|u\|_{W^{1,p}(\Omega)}\Big)\le
C_2\Big(
\|f\|_{L^p(\Omega)}+\widehat C\,\|u\|_{L^2(\Omega)}\Big).
\end{equation*}
This proves~\eqref{inequalityregolarità} when\footnote{We note that the case~$N\le p$
includes in particular the case $N=1$ in~\eqref{inequalityregolarità}.}
$N\le p$,
hence, without loss of generality, we can now assume that
\begin{equation}\label{SLMwe2.12}
N>p.
\end{equation}

Now we point out that, 
to prove~\eqref{inequalityregolarità}, it suffices to show (up to renaming constants) that
\begin{equation}\label{inequalityregolaritàSEM}
\|u\|_{W^{2,p}(\Omega)}\leq 
C\Big(\|f\|_{L^p(\Omega)}+\|u\|_{H^1(\Omega)}\Big).
\end{equation}
Indeed, since~$p\ge2$, we have that
\begin{equation}\label{inequalityregolaritàSEM:CHHA}
\|u\|_{H^2(\Omega)}\le C_\star\,\|u\|_{W^{2,p}(\Omega)}\end{equation}
for some~$C_\star>0$ depending only on~$n$, $p$ and~$\Omega$.
Hence, if~\eqref{inequalityregolaritàSEM} is satisfied we know that
\begin{equation}\label{inequalityregolaritàSEM:CH}
\|u\|_{W^{2,p}(\Omega)}<+\infty.\end{equation}
Also, using~\eqref{INY} with~$P:=2$,
$$\|u\|_{H^1(\Omega)}\le \e\|u\|_{H^2(\Omega)}+C_\e\,\|u\|_{L^2(\Omega)}.$$
This, in combination with~\eqref{inequalityregolaritàSEM}
and~\eqref{inequalityregolaritàSEM:CHHA}, gives that
\begin{eqnarray*}
\|u\|_{W^{2,p}(\Omega)}&\leq &
C\Big(\|f\|_{L^p(\Omega)}+\e\|u\|_{H^2(\Omega)}+C_\e\,\|u\|_{L^2(\Omega)}\Big)\\&\le&
C\Big(\|f\|_{L^p(\Omega)}+\e \,C_\star \,\|u\|_{W^{2,p}(\Omega)}+C_\e\,\|u\|_{L^2(\Omega)}\Big).\end{eqnarray*}
We can now choose~$\e:=\frac{1}{2C \,C_\star}$ and reabsorb~$\|u\|_{W^{2,p}(\Omega)}$
into the left-hand side, thus obtaining~\eqref{inequalityregolarità}
(and we stress that~\eqref{inequalityregolaritàSEM:CH} has been used in this reabsorption).\bigskip

Hence, our goal now is to prove~\eqref{inequalityregolaritàSEM}
and, to accomplish this,
we proceed iteratively. 
We define~$P_0:=p$ and recursively, for all~$\ell\in\N$,
\begin{equation}\label{a2lbeee23kepndlerv}
P_{\ell+1}:=\frac{N P_\ell}{N+P_\ell}.\end{equation}
One can check by induction that~$P_\ell>0$ for all~$\ell\in\N$.
As a consequence of this, one sees that
\begin{equation}\label{PA:SLPKowerf}
{\mbox{$P_\ell$ is a decreasing sequence}}\end{equation}
and, accordingly, we can define
$$ P_\infty:=\lim_{\ell\to+\infty}P_\ell.$$
By construction,
$$ P_\infty=\lim_{\ell\to+\infty}P_{\ell+1}=\lim_{\ell\to+\infty}
\frac{N P_\ell}{N+P_\ell}=\frac{N P_\infty}{N+P_\infty}$$
and therefore~$P_\infty=0$.

Consequenly, we can select~$L\in\N$, depending on~$N$ and~$p$,
which is the largest index for which~$P_{L}\ge2$, namely
\begin{equation}\label{LJSM-rfpgkrb-01}
P_{L+1}<2\le P_{L}\le\dots\le P_0=p.\end{equation}

Interestingly, since the map~$(0,+\infty)\ni r\mapsto\nu(r):= \frac{Nr}{N+r}$
is increasing, we have that
\begin{equation}\label{LJSM-rfpgkrb-02}
P_{L+1}=\nu(P_L)\ge \nu(2)=\frac{2N}{N+2}=
2-\frac{4}{N+2}>2-\frac{4}{p+2}\ge2-\frac{4}{2+2}=1.
\end{equation}
owing to~\eqref{SLMwe2.12}.

It is also useful to notice that, in view of~\eqref{SLMwe2.12} and~\eqref{PA:SLPKowerf}, for all~$\ell\in\N$ we have that~$P_{\ell+1}\le
P_0=p<N$. This allows us to invert~\eqref{a2lbeee23kepndlerv} and find that
\begin{equation}\label{a2lbeee23kepndlerv.b}
P_{\ell}=\frac{N P_{\ell+1}}{N-P_{\ell+1}}.\end{equation}

Now we set$$p_j:=P_{L+1-j}.$$ We observe that, by~\eqref{PA:SLPKowerf},
$p_j$ is an increasing sequence
and, by~\eqref{LJSM-rfpgkrb-01} and~\eqref{LJSM-rfpgkrb-02},
\begin{equation}\label{LJSM-rfpgkrb-03}p_0=P_{L+1}\in(1,2).\end{equation}
Also, using~\eqref{a2lbeee23kepndlerv.b} with~$\ell:=L-j$,
\begin{equation}\label{LJSM-rfpgkrb-03.12e023}
p_{j+1}=P_{L-j}=\frac{N P_{L+1-j}}{N-P_{L+1-j}}=\frac{N p_j}{N-p_j}.
\end{equation}

We claim that, for all~$j\in\{0,\dots,L+1\}$,
\begin{equation}\label{89:L-qe32-r}
\|u\|_{W^{2,p_j}(\Omega)}\leq C_j\Big(
\|f\|_{L^p(\Omega)}+\|u\|_{H^1(\Omega)}\Big),
\end{equation}
for some positive constant~$C_j$, depending only on~$N$, $p$, $\Omega$,
$\widetilde s$, $\overline\mu$ and~$j$.

The proof of this claim is by induction. \label{APKsqwfe-r8ngtoglk4}
First, using~\eqref{bg-92X4cv0 6m}
with~$P:=2$ we obtain that
\begin{equation}\label{89:L-qe32-r.12e4r23t}\|u\|_{W^{2,2}(\Omega)}\leq C_2\Big(
\|f\|_{L^2(\Omega)}+\|u\|_{W^{1,2}(\Omega)}\Big).\end{equation}
Furthermore, since~$p\ge2$, we can control~$\|f\|_{L^2(\Omega)}$ from above
by~$\|f\|_{L^p(\Omega)}$. On the same ground, by~\eqref{LJSM-rfpgkrb-03},
we can control~$\|u\|_{W^{2,p_0}(\Omega)}$ from above
by~$\|u\|_{W^{2,2}(\Omega)}$. These observations and~\eqref{89:L-qe32-r.12e4r23t}
give~\eqref{89:L-qe32-r} with~$j=0$, which constitutes the basis of the induction.
 
Suppose now that~\eqref{89:L-qe32-r} for some index~$j\le L$.
We consider the Sobolev exponent~$p_j^*$ of~$p_j$. Indeed,
owing to~\eqref{SLMwe2.12},
it holds that~$N>p=P_0= p_{L+1}\ge p_j$,  therefore, recalling~\eqref{LJSM-rfpgkrb-03.12e023},
\begin{equation}\label{52twgsp1o2elfi2u3e0-skjdm2prt}
p_j^*=\frac{Np_j}{N-p_j}=p_{j+1}.\end{equation}

Thus, we make use of
the Sobolev immersion of~$W^{2,p_j}(\Omega)$
into~$W^{1,p_j^*}(\Omega)$ to deduce from~\eqref{89:L-qe32-r} that
\begin{equation*} 
\|u\|_{W^{1,p_j^*}(\Omega)}\leq C_j^*\Big(
\|f\|_{L^p(\Omega)}+\|u\|_{H^1(\Omega)}\Big),
\end{equation*}for some positive constant~$C_j^*$, depending only on~$N$, $p$, $\Omega$,
$\widetilde s$, $\overline\mu$ and~$j$.

This and~\eqref{bg-92X4cv0 6m}, used here with~$P:=p_j^*$, return that
\begin{equation}\label{0LL13-23e}
\|u\|_{W^{2,p_j^*}(\Omega)}\leq C_2\Big(
\|f\|_{L^{p_j^*}(\Omega)}+\|u\|_{W^{1,p_j^*}(\Omega)}\Big)
\le \widetilde C_j\Big(\|f\|_{L^{p_j^*}(\Omega)}+
\|f\|_{L^p(\Omega)}+\|u\|_{H^1(\Omega)}\Big),
\end{equation}for some positive constant~$\widetilde C_j$, depending only on~$N$, $p$, $\Omega$,
$\widetilde s$, $\overline\mu$ and~$j$. 

We stress that~$p_{j+1}\le p_{L+1}=P_0= p$, hence~$\|f\|_{L^{p_{j+1}}(\Omega)}$
is controlled from above by~$\|f\|_{L^{p}(\Omega)}$, up to constants.
This observation, \eqref{0LL13-23e} and~\eqref{52twgsp1o2elfi2u3e0-skjdm2prt}
give~\eqref{89:L-qe32-r} for the index~$j+1$, as desired.
The proof of~\eqref{89:L-qe32-r} is thereby complete.

Using now~\eqref{89:L-qe32-r} with~$j:=L+1$
we obtain~\eqref{inequalityregolaritàSEM}, as desired.
\end{proof}

As a simple consequence of Theorem \ref{thmregolarità}
and the Sobolev Embedding Theorem, one also obtains:

\begin{corollary}\label{corregolarità}
Let~$\widetilde s\in[0,1/2)$. Assume that $f\in L^p(\Omega)$ with $p>N$. Let $u\in H^1_0(\Omega)$ be a weak solution of problem~\eqref{problemaregolarità}.

Then, $u\in C^{1,\alpha}(\Omega)$ with $\alpha=1-\frac{N}{p}$ and
there exists a constant $C=C(N,p,\Omega,\widetilde s,\overline\mu)$ such that
\begin{equation}\label{inequalityregolaritàcor}
\|u\|_{C^{1,\alpha}(\Omega)}\leq 
C(\|f\|_{L^p(\Omega)}+\|u\|_{L^2(\Omega)}).
\end{equation}
\end{corollary}

For completeness, we also mention that in the sole presence of a positive measure,
the estimates in~\eqref{inequalityregolarità} and~\eqref{inequalityregolaritàlow}
can be improved, as stated in the following result:

\begin{corollary}\label{thmregolarità:IMPRO}
Let~$\widetilde s\in[0,1/2)$.
Assume that $f\in L^p(\Omega)$ with $p\in(1,+\infty)$. 

Let $u\in H^1_0(\Omega)$ be a weak solution of the problem
\begin{equation*}
\begin{cases}
-\Delta u +\displaystyle\int_{[0,\widetilde s]}(-\Delta)^s u\,d\mu^+(s)
=f(x) & \mbox{in } \Omega,
\\
u=0  & \mbox{in } \R^N\setminus \Omega.
\end{cases}
\end{equation*}

Then, $u\in W^{2,p}(\Omega)$ and there exists a constant $C=C(N,p,\Omega,\widetilde s,\overline\mu)$ such that$$
\|u\|_{W^{2,p}(\Omega)}\leq 
C\, \|f\|_{L^p(\Omega)}.$$
\end{corollary}

The proof of Corollary~\ref{thmregolarità:IMPRO} is based on a compactness argument combined with
a suitable maximum principle. The details are presented in Appendix~\ref{thmregolarità:IMPRO:APP}

\section{Analysis of the doubly-nonlocal
first eigenvalue of problem \eqref{eigenvalueproblemmu}:
proof of Theorems~\ref{main1} and \ref{simple>0}}\label{0qspidkfewp0rthn9nf5ac01sklO}
In this section we provide the proofs of Theorems~\ref{main1} and \ref{simple>0}.
We start with the proof of Theorem \ref{main1}.

\begin{proof}[Proof of Theorem \ref{main1}]
We set
\begin{equation}\label{lambda*def}
\lambda_*:=\lim_{n\to +\infty}\lambda_{\varepsilon_n}
\end{equation}
and we claim that
\begin{equation}\label{claimlambda0}
\lambda_*=\lambda_0. 
\end{equation}
Indeed, from \eqref{lambdaen2} and \eqref{lambda0u0} we have
\[
\lambda_{\varepsilon_n}\leq \dfrac{\displaystyle\int_{[0, 1]} [u_0]^2_s \, d\mu^+(s)}{\displaystyle \int_{[0,\varepsilon_n] } [u_0]^2_s \, d\mu_n^-(s)}= \lambda_0  \dfrac{\|u_0\|^2_{L^2(\Omega)}}{\displaystyle \int_{[0,\varepsilon_n] } [u_0]^2_s \, d\mu_n^-(s)}.
\]
Hence, by passing to the limit, and using \eqref{lambda*def} and Corollary~\ref{lemma1}, we obtain that
\[
\lambda_* =\lim_{n\to +\infty}\lambda_{\varepsilon_n}\le \lambda_0 \lim_{n\to +\infty} \dfrac{\|u_0\|^2_{L^2(\Omega)}}{\displaystyle \int_{[0,\varepsilon_n] } [u_0]^2_s \, d\mu_n^-(s)} = \lambda_0,
\]
that is
\begin{equation}\label{claimlambda01}
\lambda_*\leq \lambda_0.
\end{equation}
Moreover, in virtue of~\eqref{lambdaen2} and \eqref{normadefinizione}, we see that $\lambda_{\varepsilon_n}=\|u_{\varepsilon_n}\|^2_{X^+}$.  From this and \eqref{lambda*def} we infer that $u_{\varepsilon_n}$ is bounded in $X^+(\Omega)$.
Hence, up to a subsequence, 
\begin{align}\label{deboleinX+}
&u_{\varepsilon_n}\rightharpoonup u_* \mbox{ in } X^+(\Omega)\\ \label{forteinL2}
\mbox{ and }\quad & u_{\varepsilon_n}\to u_* \mbox{ in } L^2(\Omega).
\end{align}

Also, by \eqref{lambda0u0}, \eqref{lambdaen1} and \eqref{lambdaen2},
\[
\lambda_0\leq \dfrac{\displaystyle\int_{[0, 1]} [u_{\varepsilon_n}]^2_s \, d\mu^+(s)}{\|u_{\varepsilon_n}\|^2_{L^2(\Omega)}} =\frac{\lambda_{\varepsilon_n}}{\|u_{\varepsilon_n}\|^2_{L^2(\Omega)}}.
\]
Thus, passing to the limit this expression and using \eqref{forteinL2}, \eqref{lambda*def} and Corollary \ref{cor=1}, we get
\[
\lambda_0\le\lim_{n\to +\infty} \frac{\lambda_{\varepsilon_n}}{\|u_{\varepsilon_n}\|^2_{L^2(\Omega)}} = \frac{\lambda_*}{\|u_*\|^2_{L^2(\Omega)}}=\lambda_*,
\]
that is
\[
\lambda_0 \le\lambda_*.
\]
This, together with \eqref{claimlambda01},  proves the claim in \eqref{claimlambda0}. 

Hence, \eqref{lambda*def} and \eqref{claimlambda0} establish~\eqref{statement12-PRIMO}, as desired.

We now consider the problem
\[
\begin{cases}
\displaystyle\int_{[0, 1]} (-\Delta)^s u\, d\mu^+(s) = \lambda_* u &  \mbox{ in } \Omega,\\
u=0 &\mbox{ in } \R^N \setminus \Omega.
\end{cases}
\]
We aim to show that
\begin{equation}\label{u*autofunzione}\begin{split}
&{\mbox{$u_*$ is an eigenfunction associated with $\lambda_*$, that is}}\\&
\int_{[0,1]}\langle u_*,v\rangle_s\,d\mu^+(s)
=\lambda_* \int_\Omega u_*(x) v(x)\,dx
\quad\mbox{ for any } v\in C^\infty_c(\Omega).\end{split}
\end{equation}
By \eqref{lambdaen2} we have that, for any $n\in \N$
and any~$v\in C^\infty_c(\Omega)$,
\begin{equation}\label{zserfvbhjikop}
\begin{aligned}
&\int_{[0,1]} 
\langle u_{\varepsilon_n},v\rangle_s
\,d\mu^+(s) =\lambda_{\varepsilon_n}  \int_{[0, \varepsilon_n]}
\langle u_{\varepsilon_n},v\rangle_s
\,d\mu^-_n(s) .
\end{aligned}
\end{equation}
Now, by \eqref{deboleinX+}, we infer that
\begin{equation}\label{zserfvbhjikop1}
\begin{aligned}
\lim_{n\to +\infty}\int_{[0,1]}
\langle u_{\varepsilon_n},v\rangle_s
\,d\mu^+(s)
=\int_{[0,1]}\langle u_*,v\rangle_s\,dx\,dy\,d\mu^+(s).
\end{aligned}
\end{equation}
Moreover, since $v\in C^\infty_c(\Omega)$, we can integrate by parts and get
\begin{equation}\label{zserfvbhjikop2}
\begin{aligned}&
\int_{[0, \varepsilon_n]}\langle u_{\varepsilon_n},v\rangle_s\,d\mu^-_n(s)
=\int_\Omega u_{\varepsilon_n}(x)\left(\,\int_{[0, \varepsilon_n]}(-\Delta)^s v(x) d\mu^-_n(s)\right)dx \\
&\qquad=\int_\Omega u_{\varepsilon_n}(x)\left(\,\int_{[0, \varepsilon_n]}\Big((-\Delta)^s v(x)-v(x)\Big) d\mu^-_n(s)\right)dx +\int_{\Omega}u_{\varepsilon_n}(x)v(x)\,dx.
\end{aligned}
\end{equation}
By \eqref{forteinL2}, we also have that
\begin{equation}\label{zserfvbhjikop3}
\lim_{n\to +\infty}\int_{\Omega}u_{\varepsilon_n}(x)v(x)\,dx=
\int_{\Omega}u_*(x)v(x)\,dx.
\end{equation}
In addition, by the Cauchy-Schwartz and Minkowski inequalities we obtain
\begin{equation}\label{fyvgubhijnkmol,ò}
\begin{aligned}
\int_\Omega u_{\varepsilon_n}(x)&\left(\,\int_{[0, \varepsilon_n]}\Big((-\Delta)^s v(x)-v(x)\Big) d\mu^-_n(s)\right)dx \\
&\le \|u_{\varepsilon_n}\|_{L^2(\Omega)}\left\|\,\,\int_{[0, \varepsilon_n]}\Big((-\Delta)^s v-v\Big)d\mu^-_n(s) \right\|_{L^2(\R^N)}\\
&\le \|u_{\varepsilon_n}\|_{L^2(\Omega)}\int_{[0, \varepsilon_n]}\left\|(-\Delta)^s v-v \right\|_{L^2(\R^N)}d\mu^-_n(s).
\end{aligned}
\end{equation}
Then, by the Plancherel Theorem, we infer that
\begin{equation}\label{fyvgubhijnkmol,ò2}
\begin{aligned}
\int_{[0, \varepsilon_n]}\left\|(-\Delta)^s v-v \right\|_{L^2(\R^N)}d\mu^-_n(s)
\le \int_{[0, \varepsilon_n]} \left(\,\,\int_{\R^N}||2\pi\xi|^{2s}-1||\widehat{v}(\xi)|^2\,d\xi\right)^\frac{1}{2}d\mu^-_n(s) .
\end{aligned}
\end{equation}

We also notice that, for all~$s\in[0,1]$,
\[
||2\pi\xi|^{2s}-1||\widehat{v}(\xi)|^2\leq 2(1+|2\pi\xi|^2) |\widehat{v}(\xi)|^2\in L^1(\R^N).
\]
From this and~\eqref{epsnto0}, we obtain that
$$ \lim_{n\to +\infty}\int_{[0, \varepsilon_n]} \left(\,\,\int_{\R^N}||2\pi\xi|^{2s}-1||\widehat{v}(\xi)|^2\,d\xi\right)^\frac{1}{2}d\mu^-_n(s)=0.$$
We plug this information into~\eqref{fyvgubhijnkmol,ò2} and we conclude that
$$ \lim_{n\to +\infty}
\int_{[0, \varepsilon_n]}\left\|(-\Delta)^s v-v \right\|_{L^2(\R^N)}d\mu^-_n(s)=0.$$
This and \eqref{fyvgubhijnkmol,ò} yield that
\begin{equation}\label{zserfvbhjikop4}
\lim_{n\to +\infty}\int_\Omega u_{\varepsilon_n}(x)\left(\,\int_{[0, \varepsilon_n]}\Big((-\Delta)^s v(x)-v(x)\Big) d\mu^-_n(s)\right)dx=0.
\end{equation}

Hence,  passing to the limit in \eqref{zserfvbhjikop}, in light of  \eqref{zserfvbhjikop1}, \eqref{zserfvbhjikop2}, \eqref{zserfvbhjikop3} and \eqref{zserfvbhjikop4}, we get the desired claim in \eqref{u*autofunzione}. 

We observe that~\eqref{claimlambda0}, \eqref{deboleinX+} and \eqref{u*autofunzione} yield the desired result in~\eqref{statement121P0EIDFJWEOBGML}.

Moreover, from \eqref{claimlambda0}, \eqref{forteinL2} and \eqref{u*autofunzione}, the desired result in \eqref{statement12} follows.
\end{proof}

We now present the proof of Theorem \ref{simple>0}.

\begin{proof}[Proof of Theorem \ref{simple>0}]
We first prove that
\begin{equation}\label{claimpositivo}
\begin{aligned}
\mbox{any eigenfunction}&\mbox{ associated
with the eigenvalue $\lambda_{\varepsilon_n}$ has constant sign}, \\
&\mbox{ provided that $n$ is large enough.}
\end{aligned}
\end{equation}
To accomplish this, we let $u_{\varepsilon_n}$ be an eigenfunction associated
with the eigenvalue $\lambda_{\varepsilon_n}$, normalized as in~\eqref{lambdaen1},
and we want to prove that $u_{\varepsilon_n}$ is positive if $n$ is large enough.

For this, we observe that the normalization in~\eqref{lambdaen1} entails
that
\begin{equation}\label{02eikdfv.1w2erfegrv8x82-pe}
\|u_{\varepsilon_n}\|^2_{H^1(\Omega)}=\lambda_{\varepsilon_n}.
\end{equation}

From Theorem \ref{main1} and the notation posed in~\eqref{lambda0u0}
we have that 
\begin{equation}\label{02eikdfv.1w2erfegrv8x82-pe.2erfgngKK}
{\mbox{$u_{\varepsilon_n}$ converges 
weakly to $u_0$ in $H^1(\Omega)$,}}\end{equation} where~$u_0$
is a first eigenfunction of the Laplacian in~$\Omega$.
Since~$\Omega$ is connected, we have that~$u_0$
is unique up to scalar multiplication and does not change sign.
Without loss of generality, we suppose that
\begin{equation}\label{AKLPS-1we2dPKJH}
{\mbox{$u$ is positive.}}
\end{equation}

Moreover, $u_{\varepsilon_n}$ is a solution
of~\eqref{probsuccessionefacile} with~$\lambda:=\lambda_{\varepsilon_n}$,
which can be written in the form
\begin{equation}\label{v4d0qwieejpr.12e0rfeprg}
\begin{cases}
\displaystyle \Delta u_{\varepsilon_n}+ \int_{[0, 1/4]} (-\Delta)^s u\, d\widetilde\mu_n^-(s)  =0&  \mbox{ in } \Omega,
\\
u=0    &  \mbox{ in } \R^N \setminus \Omega,
\end{cases}
\end{equation}where~$\widetilde\mu_n^-:=\lambda_{\varepsilon_n}\mu_n^-\big|_{[0,\varepsilon_n]}$.
It is important to notice that, since~$\lambda_{\varepsilon_n}$
is bounded (as a byproduct of Theorem~\ref{main1}),
we have that
\begin{equation}\label{v4d0qwieejpr.12e0rfeprg02}
{\mbox{$\widetilde\mu_n^-([0,1])$ is bounded uniformly in~$n$.}}
\end{equation}

Similarly, by \eqref{02eikdfv.1w2erfegrv8x82-pe} and the Sobolev
inequality, we find that~$\|u_{\varepsilon_n}\|_{H^1(\Omega)}$ is bounded uniformly in~$n$.

This observation, \eqref{v4d0qwieejpr.12e0rfeprg} and~\eqref{v4d0qwieejpr.12e0rfeprg02}
allow us to 
employ Corollary \ref{corregolarità}
here with~$f:=0$, $\widetilde s:=1/4$ and~$\mu:=\widetilde\mu_n^-$
and obtain estimates that are uniform in~$n$:
specifically, by~\eqref{inequalityregolaritàcor} we obtain that the sequence~$u_{\varepsilon_n}$ is bounded in $C^{1,\alpha}(\Omega)$.

Thus, $u_{\varepsilon_n}$ converges weakly to some $\tilde{u}$ in
$C^{1,\alpha}(\Omega)$. From this, we infer that $u_{\varepsilon_n}$ converges weakly to $\tilde{u}$ in $H^1(\Omega)$, so that $\tilde{u}=u_0$ almost everywhere in $\Omega$, due to~\eqref{02eikdfv.1w2erfegrv8x82-pe.2erfgngKK}. Then, by the Ascoli-Arzelà Theorem
\begin{equation}\label{zsexftgbhui9ikjmkl}
u_{\varepsilon_n} \mbox{ strongly converges to } u_0 \mbox{ in }
C^1(\Omega).
\end{equation}
In addition, by~\eqref{AKLPS-1we2dPKJH} and the Hopf Lemma, we have that
\begin{equation}\label{zsexftgbhui9ikjmkl1}
\frac{\partial u_0}{\partial \nu}(x)>0 
\quad \mbox{for any } x\in \partial \Omega.
\end{equation}
Thus, combining \eqref{zsexftgbhui9ikjmkl} and \eqref{zsexftgbhui9ikjmkl1}, we infer that there exists $n_1\in \N$ such that,
for any $n\geq n_1$,
\[
\inf_{\partial \Omega}\frac{\partial u_{\varepsilon_n}}{\partial \nu}\geq \frac{1}{2}\inf_{\partial \Omega}\frac{\partial u_0}{\partial \nu}=\frac{1}{2}\min_{\partial \Omega}\frac{\partial u_0}{\partial \nu}>0.
\]
Hence, we can apply Lemma \ref{lemmavicinobordo} to get that there exist $\delta_0>0$ and $c_0>0$ such that, for any $x\in \Omega_{\delta_0}:=\{x\in \Omega \mbox{ such that dist}(x,\partial \Omega)\leq \delta_0\}$, we have that $u_{\varepsilon_n}(x)\geq c_0\, \operatorname{dist}(x,\partial \Omega)$. This gives that
\begin{equation}\label{zsedxdfvgybhjinji}
u_{\varepsilon_n}>0 \quad \mbox{in }\Omega_{\delta_0}.
\end{equation}
Moreover, we have that $u_0>0$ in $\Omega\setminus \Omega_{\delta_0}$. 

Then, since $u_{\varepsilon_n}$ converges almost everywhere to $u_0$, we infer that there exists some $n_2\geq n_1$ such that,
for any $n\geq n_2$,
\begin{equation}\label{zsedxdfvgybhjinji1}
u_{\varepsilon_n} \geq \frac{u_0}{2}>0 \quad \mbox{in }\Omega\setminus \Omega_{\delta_0}.
\end{equation}
Combining \eqref{zsedxdfvgybhjinji} and \eqref{zsedxdfvgybhjinji1}
we get the desired claim in \eqref{claimpositivo}.

We now prove that $\lambda_{\varepsilon_n}$ is simple.
To this end, let $v_{\varepsilon_n}$ be another eigenfunction corresponding to the eigenvalue $\lambda_{\varepsilon_n}$. By \eqref{claimpositivo} we infer that there exists some $n_3\geq n_2$ such that, for any $n\geq n_3$, $v_{\varepsilon_n}$ can be taken to be positive. Indeed, if $v_{\varepsilon_n}$ is negative we can consider $-v_{\varepsilon_n}$
as an eigenfunction. Moreover, up to a normalization, we can assume that 
\[
\int_{[0, \varepsilon_n]}[v_{\varepsilon_n}]^2_s\,d\mu^-_n(s)=1.
\]
Then, from Theorem \ref{main1} we infer that \begin{equation}\label{vsjdf-COA}
{\mbox{$v_{\varepsilon_n}$ converges almost everywhere to $u_0$ in $\Omega$.}}\end{equation}

Now, we set $w_{\varepsilon_n}:=u_{\varepsilon_n}-v_{\varepsilon_n}$
and we want to prove that
\begin{equation}\label{claimsemplice}
\mbox{there exists } n_4\geq n_3 \mbox{ such that, for any } n\geq n_4,\,\,
w_{\varepsilon_n}(x) =0 \mbox{ almost everywhere in } \Omega.
\end{equation}
To prove this, we assume by contradiction that 
\[
\mbox{for any } n\geq n_3 \mbox{ there exists } \overline n\geq n,
\mbox{ such that }
w_{\varepsilon_{\overline n}}(x) \not \equiv 0 \mbox{ almost everywhere in } \Omega.
\]
Then, $w_{\varepsilon_{\overline n}}$ is also an eigenfunction corresponding to $\lambda_{\varepsilon_{\overline n}}$, so by \eqref{claimpositivo} we can assume $w_{\varepsilon_{\overline n}}>0$ if $\overline{n}$ is large enough. Thus, we can extract a subsequence
$w_{\varepsilon_{\overline n_{j}}}$ such that, for any $j\in \N$, we have that $w_{\varepsilon_{\overline n_{j}}}\not \equiv 0$ almost everywhere in $\Omega$. Now, we set 
\[
\widetilde{w}_{\varepsilon_{\overline n_{j}}}:=\dfrac{w_{\varepsilon_{\overline n_{j}}}}{\displaystyle\int_{[0, \varepsilon_{\overline n_{j}}]}[w_{\varepsilon_{\overline n_{j}}}]^2_s\,d\mu^-_n(s)}.
\]
From Theorem \ref{main1} we have that 
\begin{equation}\label{zsexftgbhui9ikjmkl.1e2rdfegr}
{\mbox{$\widetilde{w}_{\varepsilon_{\overline n_{j}}}$ converges almost everywhere to $u_0$ in $\Omega$.}}\end{equation} 

Moreover, in light of~\eqref{zsexftgbhui9ikjmkl}
and~\eqref{vsjdf-COA}, we have that $w_{\varepsilon_{\overline n_{j}}}=u_{\varepsilon_{\overline n_{j}}}-v_{\varepsilon_{\overline n_{j}}}$
converges almost everywhere to $0$ in $\Omega$, which is a contradiction
with~\eqref{zsexftgbhui9ikjmkl.1e2rdfegr}. This establishes the desired result in~\eqref{claimsemplice}.
\end{proof}

\section{Spectral properties on disconnected sets:
proof of Theorems~\ref{propcambiosegno}, \ref{ATTEHAMS},
\ref{lambdasnonsemplice}
and~\ref{lambdassemplice}}\label{sezionedisconnessi}

In this section we describe some properties of the first eigenvalue and its related eigenfunctions when $\Omega$ is a disconnected set.
We first prove that, under suitable hypotheses, if $\Omega$ in not connected then any eigenfunction associated with the first eigenvalue has to change sign,
thus establishing Theorem~\ref{propcambiosegno}.

\begin{proof}[Proof of Theorem~\ref{propcambiosegno}]
Let $u_{\overline s}$ be an eigenfunction associated with the first eigenvalue $\lambda_{\overline s}$ of problem \eqref{xfgyxcfghxcfg}, whose existence is guaranteed by Proposition~\ref{lambda1},
namely
\begin{equation}\label{xfgyxcfghxcfg1}
\lambda_{\overline s}=\min_{X^+(\Omega)} \dfrac{\displaystyle\int_{\Omega}|\nabla u_{\overline s}(x)|^2\,dx}{\displaystyle \int_{[0,\overline s)}[u_{\overline s}]_s^2\,d\mu^-(s)}.
\end{equation}
We assume by contradiction that $u_{\overline s}$ is constant in sign, and without loss of generality we suppose $u_{\overline s}\geq 0$ in $\Omega$. 

We first claim that $u_{\overline s}$ cannot be identically zero in an open set of positive measure $E$. Indeed, in this case, from \eqref{AMMU0} and \eqref{xfgyxcfghxcfg} we would have, for any $x\in E$,
\begin{eqnarray*}&&
0=-\Delta u(x)= \lambda\displaystyle \int_{(0,\overline s)} c_{N,s}\int_{\R^N} \frac{- u(x+y)-u(x-y)}{|y|^{N+2s}}\, dy\,d\mu^-(s)
+\lambda\mu^-(0)u(x)\\&&\qquad\qquad\qquad=\lambda\displaystyle \int_{(0,\overline s)} c_{N,s}\int_{\R^N} \frac{- u(x+y)-u(x-y)}{|y|^{N+2s}}\, dy\,d\mu^-(s)
<0,
\end{eqnarray*}
which is a contradiction.

Thus, we have $u_{\overline s}>0$ in $\Omega$. We now define $\widetilde u_{\overline s}\in X^+(\Omega)$ as
\[
\widetilde u_{\overline s}(x):=
\begin{cases}
-u_{\overline s}(x)   &\mbox{in }\Omega_1,
\\
u_{\overline s}(x)   &\mbox{in }\R^N\setminus \Omega_1,
\end{cases}
\] 
and observe that $u_{\overline s}=|\widetilde u_{\overline s}|$.

This implies that, for any $s\in (0,\overline{s})$,
\begin{equation}\label{6gTASK0eidkjfer} \int_{\Omega}|u_{\overline s}(x)|^2\,dx=
\int_{\Omega}|\widetilde u_{\overline s}(x)|^2\,dx\quad \mbox{ and } \quad 
\int_{\Omega}|\nabla u_{\overline s}(x)|^2\,dx=\int_{\Omega}|\nabla \widetilde u_{\overline s}(x)|^2\,dx .\end{equation}

Moreover,
\begin{eqnarray*}&&
|\widetilde u_{\overline s}(x)-\widetilde u_{\overline s}(y)|\\&&\qquad=\begin{cases}
| u_{\overline s}(x)- u_{\overline s}(y)|
&{\mbox{ if either $(x,y)\in\Omega_1^2$ or $(x,y)\in(\R^N\setminus\Omega_1)^2$}}\\
u_{\overline s}(x)+ u_{\overline s}(y)
&{\mbox{ if either $(x,y)\in\Omega_1\times(\R^N\setminus\Omega_1)$
or $(x,y)\in(\R^N\setminus\Omega_1)\times\Omega$}}
\end{cases}\\&&\qquad\le
| u_{\overline s}(x)- u_{\overline s}(y)|,
\end{eqnarray*}
with strict inequality when
either $(x,y)\in\Omega_1\times(\R^N\setminus\Omega_1)$
or $(x,y)\in(\R^N\setminus\Omega_1)\times\Omega$.

Therefore,
\[
[u_{\overline s}]_s^2<[\widetilde u_{\overline s}]_s^2 .
\] 

From this and~\eqref{6gTASK0eidkjfer}, we infer that
\[
\dfrac{\displaystyle\int_{\Omega}|\nabla \widetilde u_{\overline s}(x)|^2\,dx}{\displaystyle \int_{[0,\overline s)}[\widetilde u_{\overline s}]_s^2\,d\mu^-(s)}<
\dfrac{\displaystyle\int_{\Omega}|\nabla u_{\overline s}(x)|^2\,dx}{\displaystyle \int_{[0,\overline s)}[u_{\overline s}]_s^2\,d\mu^-(s)},
\]
which is in contradiction with the minimality of $u_{\overline s}$
in \eqref{xfgyxcfghxcfg1} and concludes the proof.
\end{proof}

Next, we show that, in disconnected domains, the nonlocal setting does not satisfy the same
spectral structure of the classical case and we establish Theorem~\ref{ATTEHAMS}.

\begin{proof}[Proof of Theorem~\ref{ATTEHAMS}]
For~$j\in\{1,2\}$, we normalize~$u_{\Omega_j}$ such that
$$  \int_{[0,1) } [u_{\Omega_j}]^2_s \, d\mu^-_n(s)=1,$$
thus obtaining that
$$\lambda(\Omega_j,\mu^-)= \int_{\Omega_j} |\nabla u_{\Omega_j}(x)|^2\,dx,$$
and we will use the short notations~$\lambda(\Omega_j):=\lambda(\Omega_j,\mu^-)$ and~$\Omega:=\Omega_1\cup\Omega_2$.

Also, for a small~$t>0$, we consider~$u_t:=u_{\Omega_1}-tu_{\Omega_2}$ and we remark that
\begin{equation}\label{HaD-1weidfkevrfmB-12wdfvkj}\begin{split}
\lambda(\Omega)&\le
\dfrac{\displaystyle \int_\Omega |\nabla u_t(x)|^2\,dx}{\displaystyle \int_{[0,1) } [u_t]^2_s \, d\mu^-_n(s)}\\&=\dfrac{\displaystyle \int_\Omega \Big(|\nabla u_{\Omega_1}(x)|^2+t^2|\nabla u_{\Omega_2}(x)|^2\Big)\,dx}{\displaystyle \int_{[0,1) }\Big( [u_{\Omega_1}]^2_s+t^2
 [u_{\Omega_2}]^2_s-2t\langle u_{\Omega_1},u_{\Omega_2}\rangle_s\Big) \, d\mu^-_n(s)}\\&=\dfrac{\lambda(\Omega_1)+O(t^2)}{1+t^2-2t\displaystyle \int_{[0,1] }\langle u_{\Omega_1},u_{\Omega_2}\rangle_s \, d\mu^-_n(s)}\\&=\lambda(\Omega_1)\left(1+2t
\int_{[0,1) }\langle u_{\Omega_1},u_{\Omega_2}\rangle_s \, d\mu^-_n(s)+O(t^2)\right).
\end{split}\end{equation}

Furthermore, since~$u_{\Omega_1}$ and~$u_{\Omega_2}$ have disjoint support,
we deduce from~\eqref{0owejdmcwedf:w2ed3rfevrfs7mxwoevi}
and~\eqref{0owejdmcwedf:w2ed3rfevrf} that
\begin{equation}\label{0eif0r94-0nc-1-n95hg20nrhguov}\begin{split}
&\int_{[0,1) }\langle u_{\Omega_1},u_{\Omega_2}\rangle_s \, d\mu^-_n(s)=
\mu(0)\int_{\R^N}u_{\Omega_1}(x)\,u_{\Omega_2}(x)\,dx\\& \qquad\qquad\qquad+
\int_{(0,1) } 
c_{N,s}\iint_{\R^{2N}}\frac{(u_{\Omega_1}(x)-u_{\Omega_1}(y))(u_{\Omega_2}(x)-u_{\Omega_2}(y))}{|x-y|^{N+2s}}\,dx\,dy
\, d\mu^-_n(s)\\&\qquad=-2\int_{(0,1) } 
c_{N,s}\iint_{\R^{2N}}\frac{u_{\Omega_1}(x)\,u_{\Omega_2}(y)}{|x-y|^{N+2s}}\,dx\,dy
\, d\mu^-_n(s)<0.\end{split}
\end{equation}

It follows from the latter inequality and~\eqref{HaD-1weidfkevrfmB-12wdfvkj} that,
if~$t$ is sufficiently small, then~$\lambda(\Omega)<\lambda(\Omega_1)$. By swapping indices,
one also finds that~$\lambda(\Omega)<\lambda(\Omega_2)$.
This proves~\eqref{0owejdmcwedf:w2ed3rfevrf019kdcX3}.

Now we prove~\eqref{0owejdmcwedf:w2ed3rfevrf9}. We argue for the sake of contradiction and we suppose that there exists~$c_1$, $c_2\in[0,+\infty)$ such
that~$u_\star:=c_1u_{\Omega_1}+c_2u_{\Omega_2}$ is an
eigenfunction corresponding to~$ \lambda(\Omega)$.
Then, in~$\Omega$, in the weak sense,
$$ -\Delta u_\star= \lambda(\Omega) \int_{[0,1)}(-\Delta)^s u_\star\,d\mu^-(s) .$$
In particular, in~$\Omega_1$ the function~$u_{\Omega_2}$ vanishes identically, therefore, recalling~\eqref{0owejdmcwedf:w2ed3rfevrfs7mxwoevi}, we see that,
in the weak sense, in~$\Omega_1$,
\begin{equation*}
\begin{split}
0&=c_1\left(\Delta u_{\Omega_1}(x)+\lambda(\Omega_1)\int_{[0,1)}(-\Delta)^s u_{\Omega_1}(x)\,d\mu^-(s)\right)\\&=
\Delta u_\star(x)+c_1\lambda(\Omega_1)\int_{[0,1)}(-\Delta)^s u_{\Omega_1}(x)\,d\mu^-(s)\\&=
-\lambda(\Omega) \int_{[0,1)}(-\Delta)^s u_\star\,d\mu^-(s)
+c_1\lambda(\Omega_1)\int_{[0,1)}(-\Delta)^s u_{\Omega_1}(x)\,d\mu^-(s)\\&=
c_1\Big(\lambda(\Omega_1)-\lambda(\Omega)\Big)\int_{[0,1)}(-\Delta)^s u_{\Omega_1}(x)\,d\mu^-(s)-c_2\lambda(\Omega)\int_{[0,1)}(-\Delta)^s u_{\Omega_2}(x)\,d\mu^-(s).
\end{split}
\end{equation*}
We consider~$u_{\Omega_1}$ as a test function for this equation,
thus concluding that
\begin{equation*}
\begin{split}
0=
c_1\Big(\lambda(\Omega_1)-\lambda(\Omega)\Big)\int_{[0,1)}[ u_{\Omega_1}
]^2_s\,d\mu^-(s)-c_2\lambda(\Omega)\int_{[0,1)}
\langle u_{\Omega_1},u_{\Omega_2}\rangle_s\,d\mu^-(s).
\end{split}
\end{equation*}

This identity and~\eqref{0owejdmcwedf:w2ed3rfevrf019kdcX3} give that
\begin{equation*}\begin{split}0\ge-c_2\lambda(\Omega)\int_{[0,1)}
\langle u_{\Omega_1},u_{\Omega_2}\rangle_s\,d\mu^-(s),
\end{split}\end{equation*}
with equality if and only if~$c_1=0$.

Then, recalling~\eqref{0eif0r94-0nc-1-n95hg20nrhguov},
\begin{equation*}\begin{split}0\ge-c_2\lambda(\Omega)\int_{[0,1)}
\langle u_{\Omega_1},u_{\Omega_2}\rangle_s\,d\mu^-(s)\ge0,
\end{split}\end{equation*}
with equality if and only if~$c_2=0$.

As a result, necessarily~$c_1=c_2=0$,
but this is in contradiction with the notion of eigenfunction.
\end{proof}

Now we show with an example that, if $\Omega$ is disconnected, then the first eigenvalue is not necessary simple, thus proving Theorem~\ref{lambdasnonsemplice}.

\begin{proof}[Proof of Theorem~\ref{lambdasnonsemplice}]
We denote by $\lambda_0$ the first Dirichlet eigenvalue for the Laplacian in~$ \Omega$. We recall that,
by Proposition~\ref{PROPD1}(v), this classical eigenvalue~$\lambda_0$ has multiplicity two and its eigenspace is the span of~$u_0$ and~$v_0$, where~$u_0$ is the positive eigenfunction associated with~$\lambda_0$ in~$B_{\frac12}(e_1)$, normalized such that~$\|u_0\|_{L^2(\R^N)}=1$,
and~$v_0(x):=u_0(-x)$ (notice that~$v_0$ is the positive eigenfunction associated with~$\lambda_0$ in~$B_{\frac12}(-e_1)$).

Now, we take a decreasing sequence $s_n$ such that $s_n\in (0,1)$ for any $n\in \N$ and
\[
\lim_{n\to +\infty}s_n=0,
\]
and consider the problem 
\begin{equation}\label{xftxcfgcvghxdfgyu}
\begin{cases}
-\Delta u=\lambda (-\Delta)^{s_n} u &\mbox{in }\Omega,
\\
u=0  &\mbox{in }\R^N\setminus \Omega.
\end{cases}
\end{equation}
Then, from Proposition~\ref{lambda1} we get that problem \eqref{xftxcfgcvghxdfgyu} admits a first positive eigenvalue $\lambda_{s_n}$, with a
corresponding eigenfunction that will be denoted by~$u_{s_{n}}$
and normalized such that~$[u_{s_{n}}]_{s_{n}}=1$.

Thus, by Theorem~\ref{main1}, there exists a subsequence~$u_{s_{n_k}}$ which converges pointwise to either~$u_0$ or~$v_0$.
Up to swapping~$u_0$ and~$v_0$, we can assume that
\begin{equation}\label{xftxcfgcvghxdfgyu3}
\lim_{k\to +\infty}u_{s_{n_k}}(x)=u_0(x)\quad \mbox{ for almost every }x\in \Omega.
\end{equation}
Now, for all $x\in \R^N$ we define $v_{s_{n_k}}(x)=u_{s_{n_k}}(-x)$ and we notice that $v_{s_{n_k}}$ is an eigenfuction associated with $\lambda_{s_{n_k}}$ such that $[v_{s_{n_k}}]_{s_{n_k}}=1$. Also, by~\eqref{xftxcfgcvghxdfgyu3},
\begin{equation}\label{xftxcfgcvghxdfgyunskaxsc}\lim_{k\to +\infty}v_{s_{n_k}}(x)=
\lim_{k\to +\infty}u_{s_{n_k}}(-x)=u_0(-x)=v_0(x)\quad \mbox{ for almost every }x\in \Omega.
\end{equation}

Suppose now, for the sake of contradiction, that $\lambda_{s_{n_k}}$ (or a subsequence of that) is simple. In this case, necessarily $u_{s_{n_k}}=\pm v_{s_{n_k}}$. From this, \eqref{xftxcfgcvghxdfgyu3} and~\eqref{xftxcfgcvghxdfgyunskaxsc} we infer that, almost everywhere,
$$ 0=\lim_{k\to +\infty}u_{s_{n_k}}\mp v_{s_{n_k}}=u_0\mp v_0.
$$
In particular, almost everywhere in~$ B_\frac{1}{2}(e_1)$, we have that~$0=u_0\mp v_0=u_0$,
which is a contradiction.
\end{proof}

Interestingly, even when $\Omega$ is disconnected,
the first eigenvalue may happen to be simple, as showcased by the example
in Theorem~\ref{lambdassemplice}. We now prove this result.

\begin{proof}[Proof of Theorem~\ref{lambdassemplice}]
We denote by $\lambda_0$ the first eigenvalue of problem
\[
\begin{cases}
-\Delta u=\lambda u &\mbox{in }\Omega,
\\
u=0  &\mbox{in }\R^N\setminus \Omega.
\end{cases}
\]
In this case $\lambda_0$ is simple, by Proposition~\ref{PROPD1}(iv), and the corresponding  eigenfunction $u_0$ can be taken such that $u_0(x)>0$ for any $x\in B_\frac{1}{2}(e_1)$, $u_0(x)=0$ for any $x\in \R^N\setminus B_\frac{1}{2}(e_1)$ and $\|u\|_{L^2(\Omega)}=1$, owing to Proposition~\ref{PROPD1}(ii).

Now, let $s_n$ be a decreasing sequence such that $s_n\in (0,1)$ for any $n\in \N$ and
\[
\lim_{n\to +\infty}s_n=0.
\]
We consider the problem 
\begin{equation}\label{xxfgxcfghcvgh}
\begin{cases}
-\Delta u=\lambda (-\Delta)^{s_n} u &\mbox{in }\Omega,
\\
u=0  &\mbox{in }\R^N\setminus \Omega.
\end{cases}
\end{equation}
In view of Proposition~\ref{lambda1} we know that problem \eqref{xxfgxcfghcvgh} admits a first positive eigenvalue $\lambda_{s_n}$.
Now, we want to prove that
\begin{equation}\label{xxfgxcfghcvgh1}
\mbox{there exists }\overline n>0 \mbox{ such that for any } n>\overline n,\,\, \lambda_{s_n} \mbox{ is simple}.
\end{equation}
Thus, we assume by contradiction that
\[
\mbox{for any }\overline n>0 \mbox{ there exists } n>\overline n \mbox{ such that } \lambda_{s_n} \mbox{ is not simple}.
\]
Then, we can extract a subsequence $s_{n_k}$ such that $\lambda_{s_{n_k}}$ is not simple for any $k\in \N$, and we denote by $u_{s_{n_k}}$ and $v_{s_{n_k}}$ two different eigenfunctions corresponding to $\lambda_{s_{n_k}}$ such that $[u_{s_{n_k}}]_{s_{n_k}}=[v_{s_{n_k}}]_{s_{n_k}}=1$.
Moreover, in light of Theorem \ref{main1} we can assume that both
$u_{s_{n_k}}$ and $v_{s_{n_k}}$ converge almost everywhere to $u_0$ in $\Omega$. Setting $w_{s_{n_k}}:=u_{s_{n_k}}-v_{s_{n_k}}$, we have that
\begin{equation}\label{xxfgxcfghcvgh2}
\lim_{k\to +\infty}w_{s_{n_k}}(x)=0 \quad \mbox{ for almost every }x\in \Omega.
\end{equation}
On the other hand, since~$u_{s_{n_k}}$ and~$v_{s_{n_k}}$ are assumed to be different, we can define
\[
\widetilde w_{s_{n_k}}:=\frac{w_{s_{n_k}}}{[w_{s_{n_k}}]_{s_{n_k}}}.
\]
From Theorem \ref{main1} we have that $\widetilde w_{s_{n_k}}$ converges almost everywhere to $u_0$ in $\Omega$, which is in contradiction with \eqref{xxfgxcfghcvgh2}. The proof of~\eqref{xxfgxcfghcvgh1} is thereby completed.
\end{proof}

\begin{appendix}

\section{Improved regularity in the presence of positive measures and proof of 
Corollary~\ref{thmregolarità:IMPRO}}\label{thmregolarità:IMPRO:APP}

The regularity result in Theorem~\ref{thmregolarità} can be sharpened when~$\mu^-\equiv0$,
since the dependence on the solution can be removed from the right-hand side of~\eqref{inequalityregolarità}
and~\eqref{inequalityregolaritàlow}. This is the content of
Corollary~\ref{thmregolarità:IMPRO}, which we prove here below:

\begin{proof}[Proof of Corollary~\ref{thmregolarità:IMPRO}]
We use the short notation~$ \|u\|_\star:=\|u\|_{L^{\min\{p,2\}}(\Omega)}$,
so that we can summarize~\eqref{inequalityregolarità} and~\eqref{inequalityregolaritàlow}
into
\begin{equation}\label{SUMMA}
\|u\|_{W^{2,p}(\Omega)}\leq 
C\,\Big( \|f\|_{L^p(\Omega)}+\|u\|_\star\Big).
\end{equation}

Now suppose, for the sake of contradiction, that Corollary~\ref{thmregolarità:IMPRO} does not hold true.
Then, there exists a sequence of (not identically vanishing) weak solutions~$u_j\in H^1_0(\Omega)$ of
\begin{equation*}
\begin{cases}
-\Delta u_j +\displaystyle\int_{[0,\widetilde s]}(-\Delta)^s u_j\,d\mu^+(s)
=f_j(x) & \mbox{in } \Omega,
\\
u_j=0  & \mbox{in } \R^N\setminus \Omega,
\end{cases}
\end{equation*}
for some~$f_j\in L^p(\Omega)$, such that\begin{equation}\label{l6bwgc90htfc9324rm:A}
\|u_j\|_{W^{2,p}(\Omega)}\ge
j\|f_j\|_{L^p(\Omega)}.\end{equation}

We let~$\widetilde u_j:=\frac{u_j}{\|u_j\|_\star}$ and~$\widetilde f_j:=\frac{f_j}{\|u_j\|_\star}$.
We remark that, by~\eqref{SUMMA},
\begin{equation}\label{oq02e-1wedo-dcocjn910d1-d} \| u_j\|_{W^{2,p}(\Omega)}\leq 
C\Big(\| f_j\|_{L^p(\Omega)}+\| u_j\|_\star\Big)
\end{equation}
and therefore, in view of~\eqref{l6bwgc90htfc9324rm:A},
\begin{eqnarray*}&&\| \widetilde f_j\|_{L^p(\Omega)}=\frac{\|f_j\|_{L^p(\Omega)}}{\|u_j\|_\star}
\le\frac{\|u_j\|_{W^{2,p}(\Omega)}}{j\,\|u_j\|_\star}
\le \frac{C\Big(\| f_j\|_{L^p(\Omega)}+\| u_j\|_\star\Big)}{j\,\|u_j\|_\star}
\\&&\quad\qquad\qquad\qquad=\frac{C\,\| \widetilde f_j\|_{L^p(\Omega)}}{j}
+\frac{C}j.
\end{eqnarray*}
Reabsorbing the term involving~$\| \widetilde f_j\|_{L^p(\Omega)}$,
the above estimate gives that, for large~$j$,
\begin{equation}\label{dr01wiedkf0tvX40} \| \widetilde f_j\|_{L^p(\Omega)}\le\frac{2C}j.\end{equation}

It also follows from~\eqref{oq02e-1wedo-dcocjn910d1-d} and~\eqref{dr01wiedkf0tvX40}
that, for large~$j$,
$$\|\widetilde u_j\|_{W^{2,p}(\Omega)}\leq 
C\Big(\|\widetilde f_j\|_{L^p(\Omega)}+\|\widetilde u_j\|_\star\Big)
=C\Big(\|\widetilde f_j\|_{L^p(\Omega)}+1\Big)\le 2C.$$
Hence, by the Rellich-Kondrachov Embedding Theorem,
up to a subsequence, we can suppose that~$\widetilde u_j$ converges strongly
in~$W^{1,p}(\Omega)$ (and, in particular, in~$H^1(\Omega)$)
and weakly in~$W^{2,p}(\Omega)$ to some~$
\widetilde u$.

We stress that, owing to the above mentioned strong convergence, $\widetilde u_j$ converges to~$\widetilde u$
in~${L^{\min\{p,2\}}(\Omega)}$ and consequently
\begin{equation}\label{CLSO.9}\| \widetilde u\|_\star=\lim_{j\to+\infty}\| \widetilde u_j\|_\star=1.\end{equation}
Also, since~$H^1_0(\Omega)$ is closed, we have that
\begin{equation}\label{CLSO}\widetilde u\in H^1_0(\Omega).\end{equation}

In this setting, 
from Lemma~\ref{nons}, up to renaming~$C$, for each~$s\in(0,1)$ we find that
\begin{eqnarray*}
\langle\widetilde u_j-\widetilde u,\phi\rangle_s\le [\widetilde u_j-\widetilde u]_s\,[\phi]_s 
\le C^2 [\widetilde u_j-\widetilde u]_1\,[\phi]_1,
\end{eqnarray*}
which is infinitesimal as~$j\to+\infty$, leading to
$$ \lim_{j\to+\infty}\int_{[0,\widetilde s]}\langle \widetilde u_j,\phi\rangle_s\,d\mu^+(s)=\int_{[0,\widetilde s]}\langle \widetilde u,\phi\rangle_s\,d\mu^+(s).$$

As a result, recalling~\eqref{dr01wiedkf0tvX40},
we conclude that, for every smooth function~$\phi$ compactly supported in~$\Omega$,
\begin{equation}\label{xsdp1icq1w-2.3}\begin{split}
0&=\lim_{j\to+\infty}\int_\Omega \widetilde f_j(x)\,\phi(x)\,dx
\\&=\lim_{j\to+\infty}\left(\int_\Omega \nabla\widetilde u_j(x)\cdot\nabla\phi(x)\,dx
+\int_{[0,\widetilde s]}\langle \widetilde u_j,\phi\rangle_s\,d\mu^+(s)\right)\\&=\int_\Omega \nabla\widetilde u(x)\cdot\nabla\phi(x)\,dx
+\int_{[0,\widetilde s]}\langle \widetilde u,\phi\rangle_s\,d\mu^+(s).
\end{split}\end{equation}

Also, by~\eqref{CLSO}, we can pick a sequence~$\phi_k$ of 
smooth functions compactly supported in~$\Omega$ and converging to~$\widetilde u$
in~$H^1(\Omega)$. We remark that, by Lemma~\ref{nons}, up to renaming~$C$, for each~$s\in(0,1)$ we have that
\begin{eqnarray*}
\langle\widetilde u,\widetilde u-\phi_k\rangle_s\le [\widetilde u]_s\,[\widetilde u-\phi_k]_s 
\le C^2 [\widetilde u]_1\,[\widetilde u-\phi_k]_1,
\end{eqnarray*}
which is infinitesimal as~$k\to+\infty$, therefore
\begin{equation}\label{xsdp1icq1w-2.31} \lim_{k\to+\infty}\int_{[0,\widetilde s]}\langle \widetilde u,\phi_k\rangle_s\,d\mu^+(s)=
\int_{[0,\widetilde s]}\langle \widetilde u,\widetilde u\rangle_s\,d\mu^+(s)=
\int_{[0,\widetilde s]}[ \widetilde u]_s^2\,d\mu^+(s).\end{equation}

We now employ~\eqref{xsdp1icq1w-2.3} with~$\phi_k:=\phi$
and utilize~\eqref{xsdp1icq1w-2.31} (as well as the fact that~$\mu^+$ is a nonnegative measure)
to see that
\begin{eqnarray*}
0&=&\lim_{k\to+\infty}\int_\Omega \nabla\widetilde u(x)\cdot\nabla\phi_k(x)\,dx
+\int_{[0,\widetilde s]}\langle \widetilde u,\phi_k\rangle_s\,d\mu^+(s)\\&=&
\int_\Omega |\nabla\widetilde u(x)|^2\,dx
+\int_{[0,\widetilde s]}[ \widetilde u]_s^2\,d\mu^+(s)\\&\ge&\int_\Omega |\nabla\widetilde u(x)|^2\,dx.
\end{eqnarray*}
As a consequence, we have that~$\widetilde u$ is necessarily constant,
and then it vanishes identically, due to~\eqref{CLSO},
but this is in contradiction with~\eqref{CLSO.9}.
\end{proof}

\section{Existence of eigenvalues and proof of Proposition~\ref{lambda1}}
The outline of proof of
Proposition~\ref{lambda1} is quite standard
and relies on a direct minimization argument, but it does require some care in dealing with the negative component of the measure at the denominator in \eqref{lambdamuOmega}. The details are provided below for the facility of the reader.
\begin{proof}[Proof of Proposition~\ref{lambda1}]
Let $I:X^+(\Omega) \to \R$ be the functional defined as
\begin{equation}\label{<awexcfbhjikmkl}
I(u):=\frac{1}{2}\int_{[0, 1]} [u]^2_s \, d\mu^+(s)=\frac{1}{2}\|u\|_{X^+}^2,
\end{equation}
and set
\[
M:=\left\{u\in X^+(\Omega)\,\, \mbox{ s.t. }\int_{[0, \overline s)} [u]^2_s \, d\mu^-(s)=1 \right\}.
\]
We first prove that there exists $u_*\in M$ such that
\begin{equation}\label{<awexcfbhjikmklclaim}
\min_{u\in M} I(u)=I(u_*).
\end{equation}
For this, let $u_k$ be a minimizing sequence for $I$ on $M$, namely
$u_k\in M$ for any $k\in \N$ and 
\begin{equation}\label{<awexcfbhjikmkl1}
\lim_{k\to+\infty}I(u_k)=\inf_{u\in M} I(u)\geq 0.
\end{equation}
Thus, $I(u_k)$ is bounded in $\R$, and in light of \eqref{<awexcfbhjikmkl} we have that $u_k$ is bounded in $X^+(\Omega)$.
Then,  there exists $u_*$ in $X^+(\Omega)$ such that, up to subsequences, $u_k$ converges weakly to $u_*$ in $X^+(\Omega)$ and $u_k$ converges to $u_*$ almost everywhere in $\R^N$.
Hence, we can apply \cite[Lemma 2.8]{DPLSVlog} and get
\[
\lim_{k\to+\infty}\int_{[0, \overline s)} [u_k]^2_s \, d\mu^-(s)=\int_{[0, \overline s)} [u_*]^2_s \, d\mu^-(s),
\]
that is $u_*\in M$. 

Moreover,  since $u_k\to u$ a.e. in $\R^N$, by the Fatou Lemma (or the lower semicontinuity of the norms)
we deduce that
\[
\lim_{k\to+\infty}I(u_k)=\lim_{k\to+\infty}\frac{1}{2}\int_{[0, 1]} [u_k]^2_s \, d\mu^+(s)\geq \frac{1}{2}\int_{[0, 1]} [u_*]^2_s \, d\mu^+(s)=I(u_*)\geq \inf_{u\in M} I(u).
\]
This combined with \eqref{<awexcfbhjikmkl1} gives 
\[
I(u_*)= \inf_{u\in M} I(u).
\]
Hence,  \eqref{<awexcfbhjikmklclaim} is proved.
Moreover, since $u_*\in M$, we have $u_* \not \equiv 0$, and thus
\[
I(u_*)=\frac{1}{2}\|u_*\|_{X^+}^2>0.
\]
Now, we claim that, for any $v\in X^+(\Omega)$,
\begin{equation}\label{<awexcfbhjikmklclaim2}
\int_{[0,1]}\langle u^*,v\rangle_s\,d\mu^+(s) 
=2I(u_*)\int_{[0,\overline s)}\langle u^*,v\rangle_s\,d\mu^-(s).
\end{equation}
To this end, we let $\varepsilon\in (-1,1)\setminus \{0\}$, $v\in X^+(\Omega)$ and we set
\[
u_\varepsilon(x):=\frac{u_*(x)+\varepsilon v(x)}{\displaystyle \int_{[0, \overline s)} [u_*+\varepsilon v]^2_s \, d\mu^-(s)}.
\]
We point out that $u_\varepsilon\in M$. Moreover, we have
\[
\|u_*+\varepsilon v\|_{X^+}^2=\|u_*\|_{X^+}^2+2\varepsilon \langle u_*, v\rangle_+ +\varepsilon^2\|v\|_{X^+}^2
\]
and
\[
\int_{[0, \overline s)} [u_*+\varepsilon v]^2_s \, d\mu^-(s)=1+2\varepsilon \langle u_*, v\rangle_- +\varepsilon^2 \int_{[0, \overline s)} [ v]^2_s \, d\mu^-(s).
\]
From this and \eqref{<awexcfbhjikmkl}, we get
\[
2I(u_\varepsilon)=\frac{2I(u_*)+2\varepsilon\langle u_*, v\rangle_+ +\varepsilon^2\|v\|_{X^+}^2}{\displaystyle 1+2\varepsilon \langle u_*, v\rangle_- +\varepsilon^2 \int_{[0, \overline s)} [ v]^2_s \, d\mu^-(s)}.
\]
Hence,
\[
\frac{2I(u_\varepsilon)-2I(u_*)}{\varepsilon}=\frac{2\left(\langle u_*, v\rangle_+ -2I(u_*)\langle u_*, v\rangle_-\right)
+\varepsilon \left(\|v\|_{X^+}^2 -2I(u_*)\displaystyle\int_{[0, \overline s)} [ v]^2_s \, d\mu^-(s)\right)}{\displaystyle 1+2\varepsilon \langle u_*, v\rangle_- +\varepsilon^2 \int_{[0, \overline s)} [ v]^2_s \, d\mu^-(s)}.
\]
Since $u_*$ is a minimizer for $I$ in $M$, by taking the limit as $\varepsilon$ goes to 0 we get \eqref{<awexcfbhjikmklclaim2}.
Then, \eqref{lambdamuOmega} follows from \eqref{<awexcfbhjikmklclaim2} by taking 
$\lambda_\mu(\Omega):=2I(u_*)$.
\end{proof}

\section{Growth from the boundary}
For the proof of Theorem \ref{simple>0}, we need the following auxiliary result, stating, roughly
speaking, that a positive interior derivative forces a function to grow at least like the distance function
(with a good control of the parameters involved). For completeness, we state and prove this result here below.

\begin{lemma}\label{lemmavicinobordo}
Let $\Omega$ be a bounded open set of class $C^1$ with interior unit normal~$\nu$. Let~$\alpha\in (0, 1)$
and~$v\in C^{1,\alpha}(\overline \Omega)$ be such that $v=0$ on $\partial \Omega$ and 
\[
\inf_{\partial \Omega}\frac{\partial v}{\partial \nu}\geq a>0.
\]
Then, there exist $\delta_0>0$ and $c_0>0$ depending on $N$,
$\Omega$, $a$ and $\|v\|_{C^{1,\alpha}(\overline \Omega)}$ such that, for any 
$x\in \Omega$ with dist$(x,\partial \Omega)\leq \delta_0$, we have that
\[
v(x)\geq c_0 \operatorname{dist}(x,\partial \Omega).
\]
\end{lemma}
\begin{proof}
Since $\Omega$ is of class $C^1$, there exists some $\delta_\Omega>0$ such that, for any $x\in \Omega$ with dist$(x,\partial \Omega)\leq \delta_\Omega$, we can uniquely
define the projection of $x$ on $\partial \Omega$. We denote by $\pi(x)\in \partial \Omega$ this projection and by $\nu(\pi(x))$ the unit normal vector in $\pi(x)$, so that
\[
x=\pi(x)+\operatorname{dist}(x,\partial \Omega)\nu(\pi(x)).
\]
Now, for any $t\in[0,1]$ we define
\[
x(t):=\pi(x)+t\operatorname{dist}(x,\partial \Omega)\nu(\pi(x)),
\]
and observe that $x(0)=\pi(x)\in \partial \Omega$.

In particular, we have
\begin{equation}\label{zserfcgyujnmkl}
v(x)=v(x(1))-v(x(0))=\operatorname{dist}(x,\partial \Omega)\int_0^1 \nabla v(x(t))\cdot \nu(\pi(x))\,dt.
\end{equation}
We now set 
\[
\delta_0:=\min\left\{\delta_\Omega, \left(\frac{a}{2\|v\|_{C^{1,\alpha}(\overline\Omega)}}\right)^\frac{1}{\alpha}\right\}.
\]
Thus, for any~$t\in[0,1]$ and any $x\in \Omega$ with dist$(x,\partial \Omega)\leq \delta_0$, we have
\[
\begin{aligned}
&\big(\nabla v(x(0))-\nabla v(x(t))\big)\cdot\nu(\pi(x))\le
|\nabla v(x(t))-\nabla v(x(0))|\\&\qquad\leq \|v\|_{C^{1,\alpha}(\overline\Omega)}|x(t)-x(0)|^\alpha 
=\|v\|_{C^{1,\alpha}(\overline\Omega)}t^\alpha(\operatorname{dist}(x,\partial \Omega))^\alpha \\
&\qquad\leq \|v\|_{C^{1,\alpha}(\overline\Omega)} \delta_0^\alpha \leq \frac{a}{2}.
\end{aligned}
\]
Thus, by the Cauchy–Schwarz inequality we infer that
\begin{equation}\label{zserfcgyujnmkl1}
\nabla v(x(t))\cdot\nu(\pi(x))\geq \nabla v(x(0))\cdot\nu(\pi(x))-\frac{a}{2}=\frac{\partial v}{\partial \nu}(\pi(x))-\frac{a}{2}\geq \frac{a}{2}.
\end{equation}
Combining \eqref{zserfcgyujnmkl} and \eqref{zserfcgyujnmkl1}, we obtain
\[
v(x)\geq \frac{a}{2}\operatorname{dist}(x,\partial \Omega).
\]
Hence, the desired result holds with $c_0:=a/2$.
\end{proof}

\section{Classical eigenvalues in disconnected domains}

For the convenience of the reader, we summarize here some useful observations on the Dirichlet eigenvalues
of the Laplacian in disconnected domains. To this end, given a bounded open set~$U$ with boundary of class~$C^1$,
we denote by~$\lambda(U)$ the first eigenvalue of the Laplacian in~$U$ with homogeneous Dirichlet condition
on~$\partial U$.

\begin{proposition}\label{PROPD1}
Let~$M\in\N\cap[1,+\infty)$.
Let~$\Omega=\Omega_1\cup\dots\cup\Omega_M$, where, for all~$j\in\{1,\dots,M\}$, the set~$\Omega_j$ is open, bounded, connected,
with boundary of class~$C^1$,
and such that~$\overline{\Omega_i}\cap\overline{ \Omega_j}=\emptyset$ if $i\neq j$. 

Suppose that
\begin{equation}\label{lammegmbmejp}
\lambda(\Omega_1)\le\lambda(\Omega_2)\le\dots\le\lambda(\Omega_M).\end{equation}
Then:
\begin{itemize}
\item[(i).] $\lambda(\Omega)=\min\{\lambda(\Omega_1),\dots,\lambda(\Omega_M)\}=\lambda(\Omega_1)$.
\item[(ii).] If~$u_\Omega$ is any eigenfunction corresponding to the eigenvalue~$\lambda(\Omega)$,
for every~$j\in\{1,\dots,M\}$ we have that either~$u_\Omega\equiv0$ in~$\Omega_j$
or~$\lambda(\Omega)=\lambda(\Omega_j)$.
\item[(iii).] If~$j_\star\in\{1,\dots,M\}$ is such that~$\lambda(\Omega)=\lambda(\Omega_{j_\star})$
and~$u_{\Omega_{j_\star}}$ is any eigenfunction corresponding to the eigenvalue~$\lambda(\Omega_{j_\star})$,
then~$u_{\Omega_{j_\star}}$, extended to zero outside~${\Omega_{j_\star}}$,
is also an eigenfunction in~$\Omega$ corresponding to the eigenvalue~$\lambda(\Omega)$.
\item[(iv).] The eigenvalue~$\lambda(\Omega)$ is simple if and only if~$\lambda(\Omega_1)<\lambda(\Omega_2)$.
\item[(v).] If~$\lambda(\Omega_1)=\lambda(\Omega_2)<\lambda(\Omega_j)$
for all~$j\in\{3,\dots,M\}$, then~$\lambda(\Omega)$ has multiplicity two and its eigenspace is the linear combination of the eigenspaces of~$\lambda(\Omega_1)$ and~$\lambda(\Omega_2)$.
\end{itemize}
\end{proposition}

\begin{proof} 
We observe that, given~$j\in\{1,\dots,M\}$,
\begin{equation}\label{AiksptD1}\begin{split}
& \lambda(\Omega)=\min_{u\in H^1_0(\Omega)\setminus\{0\}}\frac{\displaystyle\int_\Omega|\nabla u(x)|^2\,dx}{\|u\|^2_{L^2(\Omega)}}
\le\min_{{u\in H^1_0(\Omega)\setminus\{0\}}\atop{\tiny{\mbox{$u\equiv0$ in~$\R^N\setminus\Omega_j$}}}}
\frac{\displaystyle\int_\Omega|\nabla u(x)|^2\,dx}{\|u\|^2_{L^2(\Omega)}}\\&\qquad\qquad\qquad=
\min_{u\in H^1_0(\Omega_j)\setminus\{0\}}\frac{\displaystyle\int_\Omega|\nabla u(x)|^2\,dx}{\|u\|^2_{L^2(\Omega)}}
=\lambda(\Omega_j).
\end{split}\end{equation}

Also
if~$u_\Omega$ is any eigenfunction corresponding to the eigenvalue~$\lambda(\Omega)$
and~$\chi_{\Omega_j}$ is the indicator function of~$\Omega_j$,
that is
$$ \chi_{\Omega_j}(x):=\begin{cases}1&{\mbox{ if }}x\in\Omega_j,\\
0&{\mbox{ otherwise, }}\end{cases}$$
we have that~$u_\Omega\chi_{\Omega_j}\in H^1_0(\Omega_j)$ and that
$$-\Delta(u_\Omega\chi_{\Omega_j})=-\Delta u_\Omega=\lambda(\Omega)\,u_\Omega
=\lambda(\Omega)\,u_\Omega\chi_{\Omega_j},$$
showing that \begin{equation}\label{AB7}\begin{split}&{\mbox{either~$u_\Omega\chi_{\Omega_j}$ is an eigenfunction in~$\Omega_j$ corresponding to
the eigenvalue~$\lambda(\Omega)$}}\\&{\mbox{or~$u_\Omega\chi_{\Omega_j}$ vanishes identically.}}\end{split}\end{equation}

As a consequence, 
\begin{equation}\label{AiksptD}
{\mbox{either~$\lambda(\Omega)\ge\lambda(\Omega_j)$
or~$u_\Omega$ vanishes identically in~$\Omega_j$.}}\end{equation}

Thus, pick a point~$p\in\Omega$ such that~$u_\Omega(p)\ne0$.
Let~$j_p\in\{1,\dots,M\}$ be such that~$p\in\Omega_{j_p}$. We deduce from~\eqref{AiksptD}
and~\eqref{lammegmbmejp}
that~$\lambda(\Omega)\ge\lambda(\Omega_{j_p})\ge\lambda(\Omega_1)$.
This and~\eqref{AiksptD1} (used with~$j:=1$) give that~$\lambda(\Omega)=\lambda(\Omega_1)$,
which completes the proof of~(i).

Moreover, by combining~\eqref{AiksptD1} and~\eqref{AiksptD}, the claim in~(ii) plainly follows.

Regarding the claim in~(iii), the assumptions in~(iii) yield that the inequality in~\eqref{AiksptD1}
is actually an equality, whence any function attaining the minimum
in the Rayleigh quotient corresponding to~$\lambda(\Omega_j)$
also attains the minimum
in the Rayleigh quotient corresponding to~$\lambda(\Omega)$
(up to extending this function as zero outside of~$\Omega_j$).
The proof of~(iii) is thereby complete.

To prove~(iv), 
let us first suppose that~$\lambda(\Omega)$ is simple and, for the sake of
contradiction, that~$\lambda(\Omega_1)=\lambda(\Omega_2)$.
Then, by~(i), we have that~$\lambda(\Omega)=\lambda(\Omega_{j_\star})$, for all~$j_\star\in\{1,2\}$,
and then, by~(iii), if~$
u_{\Omega_{j_\star}}$ is any eigenfunction corresponding to the eigenvalue~$\lambda(\Omega_{j_\star})$,
then~$u_{\Omega_{j_\star}}\chi_{\Omega_{j_\star}}$ is
an eigenfunction in~$\Omega$ corresponding to the eigenvalue~$\lambda(\Omega)$.

The assumed simplicity of~$\lambda(\Omega)$ thus entails that there exists~$c\in\R$ such that~$ u_{\Omega_{1}}\chi_{\Omega_{1}}=c\,u_{\Omega_{2}}\chi_{\Omega_{2}}$. But this gives that~$u_{\Omega_1}$ vanishes identically
in~$\Omega_1$, which provides the desired contradiction.

To complete the proof of~(iv), let us now suppose that~$\lambda(\Omega_1)<\lambda(\Omega_2)$.
Then, by~(i) and~\eqref{lammegmbmejp}, we have that~$\lambda(\Omega)=\lambda(\Omega_1)<\lambda(\Omega_j)$
for all~$j\in\{2,\dots,M\}$.

Hence, by~(ii), any eigenfunction in~$\Omega$ corresponding to the eigenvalue~$\lambda(\Omega)$
must vanish identically in~$\Omega_j$ for all~$j\in\{2,\dots,M\}$.
Accordingly, any eigenfunction in~$\Omega$ corresponding to the eigenvalue~$\lambda(\Omega)$
must lie in~$H^1_0(\Omega_1)$
and therefore be an eigenfunction in~$\Omega_1$ corresponding to the eigenvalue~$\lambda(\Omega)$.

Since~$\Omega_1$ is connected, we know that such an eigenfunction is unique,
up to scalar multiplication. We have thereby proved that the space of eigenfunctions
in~$\Omega$ corresponding to~$\lambda(\Omega)$ is one-dimensional, thus completing the proof of~(iv).

To prove~(v), we argue as follows. By~(iii), we know that the multiplicity of~$\lambda(\Omega)$ is at least two and that the eigenfunctions~$u_{\Omega_1}$
and~$u_{\Omega_2}$ corresponding to~$\lambda(\Omega_1)$ and~$\lambda(\Omega_2)$ (in~$\Omega_1$ and~$\Omega_2$
respectively, and extended to zero outside their domains) are also eigenfunctions
corresponding to~$\lambda(\Omega)$
(and notice that~$u_{\Omega_1}$ and~$u_{\Omega_2}$ are unique up to scalar multiplication). Hence, to complete the proof of~(v), we pick any eigenfunction~$u_\Omega$ corresponding to~$\lambda(\Omega)$ and we need to show that
\begin{equation} \label{oqwdjfrptjnmuy43tghbnl1}
{\mbox{$u_{\Omega}$ is a linear combination of~$u_{\Omega_1}$ and~$u_{\Omega_2}$. }}\end{equation}

For this objective, we use~(ii) to see that~$u_\Omega\equiv0$ outside~$\Omega_1\cup\Omega_2$ and therefore
\begin{equation} \label{oqwdjfrptjnmuy43tghbnl}
u_\Omega=u_\Omega\chi_{\Omega_1}+u_\Omega\chi_{\Omega_2}.\end{equation} 

Also, by~\eqref{AB7}, we know that either~$u_\Omega\chi_{\Omega_1}$ 
vanishes identically or it is an eigenfunction associated with~$\lambda(\Omega_1)$: either way, since~$\lambda(\Omega_1)$ is simple, there exists~$c_1\in\R$ such that~$u_\Omega\chi_{\Omega_1}=c_1u_{\Omega_1}$. In the same vein, there exists~$c_2\in\R$ such that~$u_\Omega\chi_{\Omega_2}=c_1u_{\Omega_2}$. Plugging these bits of information into~\eqref{oqwdjfrptjnmuy43tghbnl}, the desired result in~\eqref{oqwdjfrptjnmuy43tghbnl1} follows.
\end{proof}

\end{appendix}

\section*{Acknowledgements} 
SD, CS and EV are members of the Australian Mathematical Society (AustMS). CS and EPL are members of the INdAM--GNAMPA.

CS acknowledges the support of the Juan de la Cierva Fellowship (grant number JDC2023-
050365-I).

This work has been supported by the Australian Laureate Fellowship FL190100081
and by the Australian Future Fellowship FT230100333.

\vfill

\end{document}